\documentclass[12pt]{article}
 \usepackage[margin=1in]{geometry} 
\usepackage{amsmath,amsthm,amssymb,amsfonts}
 \usepackage{amssymb}

\usepackage[pagewise]{lineno}
\usepackage[utf8]{inputenc}

\usepackage{hyperref}
 \usepackage{csquotes}
\usepackage[utf8]{inputenc}
\usepackage[english]{babel}
\usepackage{xcolor}

\usepackage{enumitem}

\numberwithin{equation}{section}
\newtheorem{theorem}{Theorem}[section]
\newtheorem{lemma}[theorem]{Lemma}
\newtheorem{remark}{Remark}[section]

\providecommand{\keywords}[1]
{
  \small	
  \textbf{\textit{Keywords:}} #1
}

\newcommand{\MSC}[1]{%
  \small
  \textbf{\textit{Mathematics Subject Classification:}} #1
}
\title{Global boundedness, absorbing sets and mass persistence in two-dimensional chemotaxis–Navier–Stokes systems with weakly singular sensitivity and sub-logistic sources }
\author{ 
      Minh Le\thanks{Institute for Theoretical Sciences, Westlake University,  China \texttt{(leminh@westlake.edu.cn)}} 
      \and 
       Alexey Cheskidov\thanks{Institute for Theoretical Sciences, Westlake University,  China \texttt{(cheskidov@westlake.edu.cn)}} 
    }
\date{}

\begin{document}
\maketitle

\begin{abstract}
We study a chemotaxis--fluid system with weakly singular sensitivity and a sub-logistic source
\[
\begin{cases}
n_t + u \cdot \nabla n = \Delta n - \chi \nabla \cdot (n \frac{\nabla c}{c^k}) + r n - \frac{\mu n^2}{\log^\eta(n+e)},\\
c_t + u \cdot \nabla c = \Delta c - \alpha c + \beta n,\\
u_t + u \cdot \nabla u = \Delta u - \nabla P + n \nabla \phi + f,\\
\nabla \cdot u = 0,
\end{cases}
\]
in a smoothly bounded domain $\Omega \subset \mathbb{R}^2$. We show that, under suitable conditions on the initial data, with no-flux/no-flux/Dirichlet boundary conditions, there are globally bounded classical solution. Furthermore, there is an absorbing set in the topology of $C^0(\bar{\Omega}) \times W^{1, \infty}(\Omega) \times C^0(\bar{\Omega}; \mathbb{R}^2)$. Finally, we establish the persistence of the total mass, showing that the population does not face extinction as a whole.
\end{abstract}

\\
\MSC{35B35, 35K45, 35K55, 92C15, 92C17}

\numberwithin{equation}{section}

\newtheorem{Corollary}{Corollary}[theorem]
\allowdisplaybreaks

\section{Introduction} \label{Intro}
Chemotaxis is the directed motility of cells or microorganisms in response to chemical gradients present within their environment. From a mathematical perspective, it is a rich source of nonlinear phenomena, and it has attracted considerable interest since the pioneering work of Keller and Segel in the 1970s, which introduced the first mathematical model of chemotaxis \cite{Keller}. This model, now widely known as the \emph{Keller--Segel system}, describes the interplay between cell density and chemical concentration through a system of nonlinear partial differential equations:
\begin{equation}\label{KS}
    \begin{cases}
        n_t = \Delta n - \nabla \cdot (n \nabla c), \\
        \tau c_t = \Delta c - c + n,
    \end{cases}
\end{equation}
where $ \tau \in \{0,1\} $. In this system, $ n(x,t) $ denotes the density of cells and $ c(x,t) $ the concentration of the chemical signal at position $ x $ and time $ t $. 

One of the most important features of the Keller--Segel system in two spatial dimensions is the \emph{critical mass phenomenon}. This refers to the existence of a threshold value $ m_c > 0 $ such that:
\begin{itemize}
    \item If the initial total mass $ \int_\Omega n(\cdot,t) < m_c $, the solution exists globally in time and remains bounded (see, e.g., \cite{Dolbeault, Dolbeault1, NSY}).
    \item Conversely, if $ \int_\Omega n(\cdot,t) > m_c $, then the solution may blow up in finite time (see, e.g., \cite{Nagai1, Nagai2, Nagai3, Nagai4}).
\end{itemize}

However, this delicate mass threshold behavior does not extend to higher dimensions. In \cite{Winkler_2010, Winkler-2013} it has been shown  that in dimensions three and higher, blow-up can occur regardless of how small the initial mass is.  Hence the behavior of solutions changes fundamentally as the spatial dimension increases.

To prevent such blow-up behavior, various modifications of the original model have been proposed, such as including \emph{logistic damping terms} into the population dynamics. For example, a cubic damping term of the form $ f(u) = u(1 - u)(u - a) $, with $ a \geq 0 $, was introduced and studied in \cite{MT96}, and later refined in \cite{KTTM}. A major advancement was made in \cite{Tello+Winkler}, where the authors showed that a \emph{quadratic logistic source term} $ f(u) = ru - \mu u^2 $ is sufficient to ensure global existence and boundedness of solutions in arbitrary spatial dimensions, provided the damping coefficient $ \mu $ is sufficiently large. Since then, numerous extensions have been done, such as weakening the assumptions, exploring more general systems, and introducing additional biological and physical effects. We refer the interested reader to some representative works, such as \cite{Winkler-logistic, MW2022, Tian4, WMS, Minh2, Minh3}.

When the chemotactic drift term $-\nabla\cdot(n\nabla c)$ is replaced by
\[
-\chi\nabla\cdot\left(\frac{n}{c^k}\nabla c\right),
\]
where $\chi>0$ and $k\in(0,1]$, the resulting model is referred to as a chemotaxis model with \emph{singular sensitivity}. This singular sensitivity function, particularly in the logarithmic case $k=1$, originates from the Weber--Fechner law of stimulus perception and was first introduced in the seminal works \cite{Keller, Keller+Segel}. For the case $k=1$, it was shown in \cite{Biler-1999} and \cite{Winkler+2011} that global-in-time solutions exist if $\chi\leq1$ in two dimensions, or if $\chi<\sqrt{2/N}$ in spatial dimension $N\geq3$. The global boundedness of solutions under the same condition on $\chi$ for $N\geq3$ was later established in \cite{Fujie}. These findings naturally raise the question of determining the largest possible constant $\chi_c$. To address this, in \cite{Lankeit_2016} it was shown that $\chi_c>1.016$ in two dimensions, which means that the threshold $\chi\leq1$ found in \cite{Biler-1999} is not sharp. In a further development \cite{Fujie_2018}, for the chemical equation written as
\[
\tau c_t=\Delta c-c+n,
\]
where $\tau>0$ is sufficiently small, the global existence was proved under the condition $\chi<\frac{N}{N-2}$ for $N\geq3$. Motivated by this result, the authors conjectured that $\chi_c=\frac{N}{N-2}$. Without requiring the smallness of $\tau$, it was proven in \cite{Ahn+Kang+Lee} that $\chi_c\geq\frac4N$ for $N\geq2$. This further implies that the bound $\chi<\sqrt{2/N}$ given in \cite{Winkler+2011} is not optimal when $3\leq N<8$. In line with this, the recent work \cite{Le_2025} establishes that the bound $\chi<\sqrt{2/N}$ is indeed not optimal for all dimensions $N\geq3$.

Let us now shift our attention to the Keller-Segel system with weakly singular sensitivity ($k \in (0,1)$). It was shown in \cite{NaSe3} that when $N=2 $ then solutions to the parabolic-elliptic system are global and bounded given that initial conditions are radial symmetric while $N \geq 3$ finite-time blow-up solutions may occur for all $k>0$. This naturally leads to the question of whether the inclusion of a logistic source term of the form $ rn - \mu n^2 $ can prevent such blow-up behavior. In \cite{zh}, a positive answer was provided in two dimensions, under the condition that the logistic damping coefficient $ \mu $ is sufficiently large. This largeness condition on $ \mu $ was later removed in \cite{Minh6}, where it was further demonstrated that even a \emph{sub-logistic source} of the form 
\[
f(u) = rn - \frac{\mu n^2}{\log^\gamma(n + e)} \quad \text{with } \gamma \in (0,1)
\]
can prevent blow-up of solutions. In higher-dimensional domains, the global boundedness of solutions to the parabolic-elliptic system was established in \cite{Kurt} under the assumptions that $ k < \frac{1}{2} + \frac{2}{N} $ and $ \mu $ is sufficiently large. This result was subsequently improved in \cite{minh-kurt}, where it was shown that a quadratic logistic damping term $ -\mu u^2 $ with large enough $ \mu $ ensures the global boundedness of solutions for all $ k \in (0,1) $. Similar effects have been observed for the fully parabolic system as well. In \cite{minh-preprint-1}, it was shown that logistic damping prevents blow-up in two-dimensional domains, and in higher dimensions this was extended in \cite{minh-preprint-2}, provided that $ \mu $ is sufficiently large.

In natural aquatic environments, microorganisms live suspended in fluid, which exposes them to a range of physical effects, including buoyancy, gravity, and transport by the surrounding flow. One of the first mathematical models to capture the coupling between chemotactic motion and ambient flow was proposed in the influential work of \cite{Tuval2005}, which described the motion of aerobic bacteria in sessile drops. In this paper, we study chemotaxis with singular sensitivity and its interaction with the surrounding fluid. Specifically, we consider the following chemotaxis-Navier-Stokes system with weakly singular sensitivity, posed in a smooth bounded domain $\Omega \subset \mathbb{R}^2$:
\begin{equation} \label{1}
    \begin{cases}
        n_t + u \cdot \nabla n = \Delta n - \chi \nabla \cdot \left ( \dfrac{n}{c^k} \nabla c \right ) + r n - \frac{\mu n^2}{\log^\eta(n+e)}, \qquad & x \in \Omega, \, t>0, \\ 
        c_t + u \cdot \nabla c = \Delta c - \alpha c + \beta n, \qquad  & x \in \Omega, \, t>0, \\
        u_t + u \cdot \nabla u = \Delta u - \nabla P + n \nabla \phi + f,\qquad  & x \in \Omega, \, t>0,\\
        \nabla \cdot u = 0, \qquad  & x \in \Omega, \, t>0,
    \end{cases}
\end{equation}
where $r, \mu, \alpha, \beta > 0$ are fixed positive parameters. In this system, $u = u(x,t)$ and $P = P(x,t)$ denote the fluid velocity and pressure, respectively, while $\phi = \phi(x)$ represents the gravitational potential, and $f = f(x,t)$ stands for the external force acting on the fluid.

\textbf{Main results.} The main goal of this paper is to show that adding a logistic source guarantees the existence of globally bounded classical solutions to the two-dimensional chemotaxis–fluid system with weakly singular sensitivity \eqref{1}. We also show the existence of an absorbing set in the
$C^{0}(\bar{\Omega}) \times W^{1, \infty}(\Omega) \times C^0(\Omega; \mathbb{R}^2)$-topology. To prove the existence of a global attractor, one would normally use a bootstrap argument to obtain compactness of the absorbing set. This approach, however, requires a uniform positive lower bound on the chemical concentration. Unfortunately the additional complexity introduced by the fluid makes such a bound out of reach in the present setting. We can nevertheless prove that the total mass of solutions persists, which rules out extinction of the population. In the absence of fluid ($u \equiv 0$) case, this mass persistence is enough to give a positive lower bound on the chemical concentration, but in the fully coupled system it is not.

Before presenting our main results, let us first specify precise boundary and initial conditions. System \eqref{1} is supplemented with the homogeneous Neumann boundary conditions for $n$ and $c$, and the no-slip condition for the fluid velocity $u$:
\begin{align} \label{bdry}
    \frac{\partial n}{ \partial \nu}= \frac{\partial c}{ \partial \nu} = 0, \quad \text{and } u = 0 \quad \text{for } x \in \partial \Omega \text{ and } t > 0,
\end{align}
as well as the following initial conditions:
\begin{align*} 
    n(x,0) = n_0(x), \quad c(x,0) = c_0(x), \quad u(x,0) = u_0(x), \quad x \in \Omega.
\end{align*}
Throughout this paper, we impose the following regularity assumptions on the initial data:
\begin{equation} \label{initial'}
    \begin{cases}
        n_0 \in C(\Bar{\Omega}), \quad n_0 \geq 0 \text{ in } \Bar{\Omega}, \quad \text{and } n_0 \not\equiv 0, \\
        c_0 \in W^{1, \infty}(\Omega), \quad c_0 > 0 \text{ in } \Bar{\Omega}, \\
        u_0 \in D(A^\gamma) \quad \text{for some } \gamma \in \left( \dfrac{1}{2}, 1 \right),
    \end{cases}
\end{equation}
where $A:= - \mathcal{P} \Delta $ denotes the Stokes operator in  {$L^2_\sigma(\Omega) := \overline{\left\{\varphi\in C^\infty_c(\Omega;\mathbb R^2) \,\middle|\, \nabla\cdot\varphi=0 \right\}
}^{\,L^2(\Omega;\mathbb R^2)},$ }with $ \mathcal{P}$ standing for the Helmholtz projection on $L^2(\Omega; \mathbb{R}^2)$. Furthermore, given potential function $\phi$ and the source term $f$ in \eqref{1} satisfy
\begin{align} \label{phi}
   \phi \in W^{2, \infty}(\Omega),   
\end{align}
and 
\begin{align} \label{f}
    f \in C^1 \left ( \Bar{\Omega}\times [0, \infty) \right )\cap L^\infty \left ( \Omega \times (0,\infty) \right ).
\end{align}

We now present our first main result concerning the existence of global bounded classical solutions to \eqref{1}. 
\begin{theorem} \label{thm1}
Let $\Omega \subset \mathbb{R}^2$ be a bounded domain with smooth boundary, let $r,\mu,\alpha,\beta,\chi$ be positive constants, and let $\eta,k\in(0,1)$. Suppose that $\phi$ and $f$ satisfy the conditions \eqref{phi} and \eqref{f}, respectively. Then
    the system \eqref{1} under the boundary conditions \eqref{bdry} and initial conditions \eqref{initial'} possesses a global classical solution $(n,c,u,P)$ such that 
    \begin{equation*}
        \begin{cases}
            n \in C^0 \left ( \bar{\Omega}\times [0,\infty) \right ) \cap C^{2,1}\left ( \bar{\Omega}\times (0,\infty) \right ), \\
             c \in C^0 \left ( \bar{\Omega}\times [0,\infty) \right ) \cap C^{2,1}\left ( \bar{\Omega}\times (0,\infty) \right ), \\ 
             u \in C^0 \left ( \bar{\Omega}\times [0,\infty) \right ) \cap C^{2,1}\left ( \bar{\Omega}\times (0,\infty) \right ) \quad \text{and } \\
             P \in C^{1,0}\left ( \bar{\Omega} \times (0, \infty)\right ).
        \end{cases}
    \end{equation*}
    Moreover, $n>0$, $c>0$ in $\bar{\Omega}\times (0, \infty)$, and the solution is globally bounded in the sense that
    \begin{align*}
        \sup_{t>0} \left \{ \left \| n(\cdot,t) \right \|_{L^\infty(\Omega)}+ \left \| c(\cdot,t) \right \|_{W^{1,\infty}(\Omega)}+ \left \| u(\cdot,t) \right \|_{L^\infty(\Omega)}  \right \} <\infty .
    \end{align*}
\end{theorem}

\begin{remark}
When $u \equiv 0$, this theorem improves Theorem~1.1 in \cite{minh-preprint-1} and Theorem~1.3 in \cite{Minh_JDE2} by showing that the sub-logistic source indeed ensures the existence of global bounded classical solutions. 
\end{remark}

Next, we prove the existence of an absorbing set.

\begin{theorem} \label{thm2}
Under the assumptions of Theorem \ref{thm1}, the system \eqref{1} possesses a bounded absorbing set in 
$C^{0}(\bar{\Omega}) \times W^{1, \infty}(\Omega) \times C^{0,\theta}(\bar{\Omega})$ for some $\theta \in (0,1)$. 
More precisely, there exists a constant $C>0$, independent of the initial data, such that
\[
\limsup_{t \to \infty} \Biggl\{ \left\| n(\cdot,t) \right\|_{C^{0}(\bar{\Omega})}
+ \left\| c(\cdot,t) \right\|_{W^{1,\infty}({\Omega})}
+ \left\| u(\cdot,t) \right\|_{ C^{0,\theta}(\bar{\Omega})} \Biggr\} \leq C.
\]

\end{theorem}
\begin{remark}
The existence of a global attractor in $C^{0}(\bar{\Omega}) \times W^{1,\infty}(\Omega) \times C^{0,\theta}(\bar{\Omega})$ for some $\theta \in (0,1)$ for system \eqref{1} remains an open problem. Nevertheless, the framework developed in \cite{Alex} may be applicable to proving the existence of a global attractor in a weaker topological space.
\end{remark}
Our final main result establishes the persistence of the total population of solutions.

\begin{theorem}\label{thm3}
Under the assumptions of Theorem \ref{thm1}, there exist constants
$M_*,M^*>0$, possibly depending on the initial data, such that
\[
M_* \leq \int_\Omega n(x,t)\,dx \leq M^* \qquad \text{for all } t \geq0.
\]
Moreover, there exist constants $\underline M,\overline M>0$, independent
of the initial data, such that
\[
\underline M \leq \liminf_{t\to\infty}\int_\Omega n(x,t)\,dx \leq \limsup_{t\to\infty}\int_\Omega n(x,t)\,dx
\leq \overline M.
\]
More precisely, if $\overline m$ denotes the unique positive solution of
\[
\frac{\mu\overline m}{\log^\eta(\overline m+e)}=r,
\]
then one may take
\[
\overline M=|\Omega|\overline m \qquad\text{and} \qquad M^*=\max\left\{\int_\Omega n_0,\overline M\right\}.
\]
\end{theorem}

\begin{remark}
 In the absence of fluid coupling (i.e., when $u \equiv 0$), the positive lower bound for $\int_\Omega n(\cdot,t) \, dx$, together with Lemma 3.1 in \cite{Fujie_2018}, yields a uniform lower bound for $c$. This result resolves the question posed in Remark 1.7 in \cite{Minh_JDE2}. Consequently, using this lower bound together with bootstrapping arguments, it is possible to show compactness of the absorbing set, thereby establishing the existence of a global attractor. In the fully parabolic chemotaxis--fluid system, however, Lemma 3.1 in \cite{Fujie_2018} is no longer applicable, and hence such a lower bound for $c$ cannot be derived.
\end{remark}

\textbf{Challenges and main ideas. }The main difficulties in showing global boundedness of solutions, as well as the
existence of an absorbing set, arise primarily from the singular behavior of the chemical
sensitivity near zero and its coupling with the fluid interaction.

First, the presence of a sub-logistic source poses a significant challenge, since it is
unclear whether the spatio-temporal estimate
\[
\sup_{t\in (0,T_{\rm max}-\tau )} \int_t^{t+\tau} \int_\Omega
\frac{n^2(\cdot,s)}{\log^\eta (n(\cdot,s)+e)} \, dx \,ds < \infty,
\]
where $T_{\rm max} \in (0, \infty]$ is the maximal existence time specified in Lemma \ref{local} and $\tau = \min\left\{1, \frac{T_{\rm max}}{2}\right\}$,
can be exploited to deduce boundedness of
$\int_\Omega |\nabla c(\cdot,t)|^2$ and $\int_\Omega |u(\cdot,t)|^2$.
To address this issue, motivated by the main idea in \cite{Winkler_SIAM}, we use
a Moser--Trudinger type inequality (see Lemma~\ref{Moser-Trudinger}) to obtain boundedness of the $L^2$ norm of the fluid velocity $u$.

Second, in contrast to chemotaxis systems with singular sensitivity but without fluid
interaction, it remains completely open whether a positive lower bound for the chemical
concentration can be obtained in the presence of fluid coupling. Consequently, standard
techniques commonly used in the fluid-free setting fail to apply here.

To overcome this obstacle, inspired by the approach in \cite{minh-preprint-1},
we use the following energy functional:
\begin{align*}
    F(t) := \int_\Omega n \log n - \frac{1}{3} \int_\Omega n \log c
    + \frac{1}{2} \int_\Omega |\nabla c|^2, \qquad t > 0.
\end{align*}
We first show that $F(t)$ is bounded, and then prove the boundedness of 
\begin{align*}
    G(t) := \int_\Omega n^p + \int_\Omega \frac{n^p}{c^q}
    + \int_\Omega |\nabla c|^{2p}, \qquad t > 0,
\end{align*}
for all $p > 1$ and $0 < q < p - 1$.

With these estimates at hand, we apply a standard Neumann heat semigroup approach to obtain the existence
of an absorbing set.

Finally, to prove the persistence of the total population, we exploit the uniform in time boundedness of $n$ in $C^{0,\lambda}(\overline{\Omega})$ (with $\lambda \in (0,1)$) to derive a differential inequality of the form
\begin{equation*}
    m'(t) \ge r\, m(t) - C\, m^{1+\theta}(t), \qquad t > 0,
    \label{eq:differential-inequality}
\end{equation*}
where
\begin{equation*}
    m(t) := \frac{1}{|\Omega|} \int_{\Omega} n(x,t) \, dx,
\end{equation*}
and $C > 0$, $\theta \in (0,1)$, which then yields a positive lower bound for $m(t)$ for all $t > 0$.

\textbf{Plan of the paper. } In Section~\ref{S2}, we briefly recall the local well-posedness result and introduce several important inequalities that will be used frequently in the subsequent sections. In Section~\ref{S3}, we establish a number of essential estimates for solutions to \eqref{1}, which serve as a foundation for Section~\ref{S4} — the central part of the paper — where we prove the $ L \log L $ boundedness result. Next, we obtain an $L^p$ boundedness result for solutions in Section \ref{S4'}. Finally, in Section~\ref{S5}, we prove our main results, including the existence of globally bounded classical solutions, the existence of a bounded absorbing set, and the mass persistence of solutions.

\section{ Preliminaries} \label{S2}
In this section, we establish the local wellposedness of solutions as well as provide several useful inequalities. We will begin with a local existence result as stated in the following lemma.

\begin{lemma} \label{local}
     Let $\Omega \subset \mathbb{R}^2$ be a bounded domain with smooth boundary, and let $r, \mu, \alpha, \beta, \chi  $ be positive constants and $\eta , k \in (0,1)$. Assume that $(n_0,c_0,u_0)$, $\phi$ and $f$ satisfy the conditions \eqref{initial'}, \eqref{phi} and \eqref{f} respectively. Then there exist $T_{\rm max} \in (0,\infty]$ and a quadruple of functions $(n,c,u,P)$ complying with
    \begin{equation} \label{local-1}
        \begin{cases}
            n \in C^0 \left ( \bar{\Omega}\times [0,T_{\rm max}) \right ) \cap C^{2,1}\left ( \bar{\Omega}\times (0,T_{\rm max}) \right ), \\
             c \in C^0 \left ( \bar{\Omega}\times [0,T_{\rm max}) \right ) \cap C^{2,1}\left ( \bar{\Omega}\times (0,T_{\rm max}) \right ), \\ 
             u \in C^0 \left ( \bar{\Omega}\times [0,T_{\rm max}) \right ) \cap C^{2,1}\left ( \bar{\Omega}\times (0,T_{\rm max}) \right ) \quad \text{and } \\
             P \in C^{1,0}\left ( \bar{\Omega} \times (0, T_{\rm max})\right ),
        \end{cases}
    \end{equation}
  such that up to addition of constants to $P$, $(n,c,u,P)$ solves \eqref{1} under boundary conditions \eqref{bdry} uniquely on $(0,T_{\rm max})$. Moreover, $n(x,t)>0$ and
    \[
        c(x,t)\geq e^{-\alpha t}\inf_{x\in\bar\Omega}c_0(x)
        \qquad\text{for all }(x,t)\in\bar\Omega\times(0,T_{\rm max}),
    \]
    and if $T_{\rm max}<\infty$, then 
    \begin{align} \label{ext}
        \limsup_{t\to T_{\rm max}} \left \{ \left \| n(\cdot,t) \right \|_{L^\infty(\Omega)}+ \left \| c(\cdot,t) \right \|_{W^{1,\infty}(\Omega)}+ \left \| A^\gamma u(\cdot,t) \right \|_{L^2(\Omega)}  \right \} =\infty \quad.
    \end{align}
\end{lemma}
\begin{proof}
   The existence of solutions in $\Omega \times (0,T_{\rm max})$ with some $T_{\rm max}\in (0,\infty]$ and the extensibility criterion \eqref{ext} can be proven by applying a standard fixed point argument as used in \cite{Winkler2012}[Lemma 2.1] and \cite{Lankeit_2016}[Theorem 2.3(i)]. Furthermore, by the comparison principle,
    \[
        n(x,t)>0,\qquad    c(x,t)\geq e^{-\alpha t}\inf_{x\in\bar\Omega}c_0(x)
        \quad\text{for all }(x,t)\in\bar\Omega\times(0,T_{\rm max}).
    \]
    To prove uniqueness, let us assume that $(n_1,c_1,u_1,P_1)$ and
    $(n_2,c_2,u_2,P_2)$ are solutions to \eqref{1} in $\Omega\times(0,T_{\rm max})$ satisfying \eqref{local-1}. For any $T\in(0,T_{\rm max})$, we have
    \[
        c_i(x,t)\geq m_*:=e^{-\alpha T}\inf_{x\in\bar\Omega}c_0(x)
        \qquad\text{for all }(x,t)\in\bar\Omega\times(0,T],\quad i=1,2.
    \]
    Using this lower bound for $c_i$ and following arguments similar to those in
    \cite{Winkler2012}[Lemma 2.1] and \cite{KMT_2015}[Proposition 3.1], we deduce that
    $(n_1,c_1,u_1)=(n_2,c_2,u_2)$ in $\Omega\times(0,T]$. Therefore, solutions to \eqref{1} are unique up to addition of constants to $P$. 
\end{proof}
Throughout the remainder of this work, we denote $(n, c, u, P)$ as a classical solution to system~\eqref{1} in the domain $\Omega \times (0, T_{\max})$ as specified in Lemma~\ref{local}. The next two lemmas provide elementary yet very useful differential inequalities, which allows us to exploit the spatial-temporal information of the solutions.
\begin{lemma}\label{ODI}
    Let $a>0$, $t_0\in\mathbb R$, $T\in(t_0,\infty]$, and suppose that the nonnegative function
    $h\in L^1_{\rm loc}([t_0,T))$ has the property that there exist $\tau\in(0,T-t_0)$ and $b>0$ such that
    \[
    \frac1{\tau}\int_t^{t+\tau}h(s)\,ds \leq b     \qquad\text{for all }t\in[t_0,T-\tau).
    \]
    Assume that $y\in C^0([t_0,T))\cap C^1((t_0,T))$ satisfies
    \[
        y'(t)+ay(t)\leq h(t) \qquad\text{for all }t\in(t_0,T).
    \]
    Then
    \[
        y(t)\leq y(t_0)e^{-a(t-t_0)} +\frac{b\tau}{1-e^{-a\tau}}  \qquad\text{for all }t\in[t_0,T).
    \]
\end{lemma}

\begin{proof}
     For a detailed proof, we refer the reader to \cite[ Lemma 3.4]{Winkler+2019}.
\end{proof}

 \begin{lemma} \label{ODI'}
   Let $a\in\mathbb R$ and $T\in(a,\infty]$. Assume that
   $y:[a,T)\to[0,\infty)$ is absolutely continuous and that
   \[
        \int_t^{t+\tau}y(s)\,ds\leq L_1  \qquad\text{for all }t\in(a,T-\tau),
   \]
   where $\tau=\min\left\{1,\frac{T-a}{2}\right\}$ and $L_1>0$. Suppose further that
   \[
        y'(t)\leq h(t)y(t)+g(t) \qquad\text{for a.e. }t\in(a,T),
   \]
   where $h$ and $g$ are nonnegative continuous functions on $[a,T)$ satisfying
   \[
        \int_t^{t+\tau}h(s)\,ds\leq L_2, \qquad \int_t^{t+\tau}g(s)\,ds\leq L_3
        \qquad\text{for all }t\in(a,T-\tau).
   \]
   Then
   \[
        y(t)\leq \frac{L_1e^{L_2}}{\tau}+L_3e^{L_2} \qquad\text{for all }t\in(a+\tau ,T).
   \]
\end{lemma}
    
\begin{proof}
   The proof can be found in Lemma 2.4 in \cite{minhJMFM}.
\end{proof}

The following lemma provides a modified version of the Gagliardo--Nirenberg interpolation inequality, which is used to establish $ L^p $ boundedness of solutions without requiring a largeness condition on $ \mu $. 

\begin{lemma}\label{C52.ILGN} Assume that $\Omega \subset \mathbb{R}^2$ is a bounded domain with smooth boundary  and $m>0$, $\sigma>\xi \geq 0$. For each $\varepsilon>0$, there exists $C=C(\varepsilon,\xi,\sigma)>0$ such that the following inequality holds
    \begin{align}\label{C52.ILGN.1}
        \int_\Omega \phi^ {m+1} \ln^{\xi}(\phi+e) \leq \varepsilon \left ( \int_\Omega \phi \ln^{\sigma}(\phi+e) \right ) \left ( \int_\Omega |\nabla \phi^{\frac{m}{2}}|^2 \right  ) +\varepsilon \left ( \int_\Omega  \phi \right )^m \left ( \int_\Omega \phi \ln^{\sigma}(\phi+e) \right ) +C,
    \end{align}
     for all positive function $\phi \in C^1 (\bar{\Omega})$.
\end{lemma}
\begin{proof}
    The proof can be found in \cite[Lemmas 2.4 and 2.5]{Minh5}.
\end{proof}

Next, we recall from  \cite{Winkler_SIAM}[Lemma 2.2] following functional inequality which will be later applied to obtain $L^2$ boundedness for $u$ in Lemma \ref{L2u}.

\begin{lemma} \label{Moser-Trudinger}
    Let $\Omega \subset \mathbb{R}^2$ be a bounded domain with smooth boundary. Then for all $\varepsilon>0$ there exists $M=M(\varepsilon,\Omega)>0$ such that if $0  \not\equiv \phi \in  C^0(\bar{\Omega}) $ is nonnegative and $\psi \in W^{1,2}(\Omega)$, then for each $a>0$
    \begin{align}
        \int_\Omega \phi |\psi| \leq \frac{1}{a}\int_\Omega \phi \ln \frac{\phi}{\bar{\phi}} + \frac{(1+\varepsilon)a}{8\pi}\cdot \left \{ \int_\Omega \phi \right \} \cdot \int_\Omega |\nabla \psi|^2+Ma \cdot \left \{ \int_\Omega \phi \right \} \cdot \left \{ \int_\Omega |\psi| \right \}^2+ \frac{M}{a}\int_\Omega \phi,
    \end{align}
    where $\bar{\phi}:= \frac{1}{|\Omega|}\int_\Omega \phi $.
\end{lemma}

The following lemma allows us to address the singularity of $c$ near zero, even without assuming a positive lower bound for $c$. When $p \ge 3$, the result follows directly from Proposition 1.3 in \cite{Kurt+Shen} or Proposition 3.1 in \cite{Kurt}, while a more general proof valid for all $p>1$ is given in \cite{minh-preprint-2}[Lemma 2.4].

\begin{lemma} \label{LK-1}
Let $\Omega \subset \mathbb{R}^n$ ($n \geq 2$) be a bounded domain with smooth boundary. For any $p > 1$, there exist positive constants $C_1$ and $C_2$, depending only on $p$, such that for all positive functions $w \in C^2(\overline{\Omega})$ satisfying $\frac{\partial w}{\partial \nu} = 0$ on $\partial \Omega$, the inequality
\[
\int_\Omega \frac{|\nabla w|^{2p}}{w^p} \,dx \;\leq\; C_1 \int_\Omega |\Delta w|^p \,dx \;+\; C_2 \int_\Omega w^p \,dx
\]
holds.
\end{lemma}

\section{A Priori Estimates} \label{S3}
In this section we obtain a series of a priori estimates for solutions to prepare for the proof of the main results. We start with a uniform $L^1$ bound and a spatial-temporal $L^2$ bound.
\begin{lemma} \label{L1}
   There exist positive constants $C_1$, $C_2$, and $C_3$ such that 
\begin{align} \label{L1-1}
    \int_\Omega n(\cdot,t) \leq C_1 \qquad \text{for all }t\in (0,T_{\rm max}),
\end{align}
and 
\begin{align} \label{L1-1'}
    \int_\Omega c(\cdot,t) \leq C_2 \qquad \text{for all }t\in (0,T_{\rm max}),
\end{align}
as well as
\begin{align} \label{L1-2}
    \int_t^{t+\tau}\int_\Omega \frac{n^2(\cdot,s)}{\log^\eta(n(\cdot,s)+e)}\,dx\,ds
    \leq C_3 \qquad  \text{for all }t\in (0,T_{\rm max}-\tau),
\end{align}
where $\tau = \min \left \{1, \frac{T_{\rm max}}{2} \right \}$. 
If $T_{\rm max} = \infty$, then there exist $T>0$ and positive constants
$C_4$, $C_5$, and $C_6$, with $C_4,C_5,C_6$ independent of the initial data, such that
\begin{align} \label{L1-3}
    \int_\Omega n(\cdot,t) \leq C_4 \qquad \text{for all }t>T,
\end{align}
and 
\begin{align} \label{L1-3'}
    \int_\Omega c(\cdot,t) \leq C_5 \qquad \text{for all }t>T,
\end{align}
as well as
\begin{align} \label{L1-4}
    \int_t^{t+1}\int_\Omega \frac{n^2(\cdot,s)}{\log^\eta(n(\cdot,s)+e)}\,dx\,ds
    \leq C_6 \qquad  \text{for all }t>T .
\end{align}
\end{lemma}

\begin{proof}
First, there exists $K_1>0$, independent of the initial data, such that 
\[
   (r+1) \int_\Omega n \leq \frac{\mu }{2} \int_\Omega \frac{n^2}{\log^\eta(n+e)} +K_1 \qquad \text{for all }t\in (0,T_{\rm max}).
\]
Integrating the first equation of \eqref{1} over $\Omega$ and using this inequality yields
\begin{align} \label{L1.1}
    \frac{d}{dt}\int_\Omega n +\int_\Omega n
    &= (r+1)\int_\Omega n- \mu \int_\Omega \frac{n^2}{\log^\eta(n+e)} \notag \\
    &\leq -\frac{\mu }{2}\int_\Omega \frac{n^2}{\log^\eta(n+e)} +K_1
    \qquad \text{for all }t\in (0,T_{\rm max}).
\end{align}
By Gronwall's inequality,
\begin{align} \label{L1.2}
    \int_\Omega n(\cdot,t) \leq e^{-t}\int_\Omega n_0+K_1(1-e^{-t})
    \leq C_1:=\max\left\{\int_\Omega n_0,K_1\right\} \qquad \text{for all }t\in (0,T_{\rm max}).
\end{align}
   Integrating the second equation over $\Omega$ and using \eqref{L1.2}, we obtain 
\begin{align}\label{L1.2'}
   \frac{d}{dt} \int_\Omega c +\alpha \int_\Omega c  = \beta \int_\Omega n
   \leq \beta C_1 \qquad \text{for all }t\in (0,T_{\rm max}).
\end{align}
By Gronwall's inequality,
\begin{align}
    \int_\Omega c(\cdot,t) \leq e^{-\alpha t}\int_\Omega c_0
    + \frac{\beta C_1}{\alpha}(1-e^{-\alpha t})
    \leq C_2:=\max\left\{\int_\Omega c_0,\frac{\beta C_1}{\alpha}\right\} \qquad \text{for all }t\in (0,T_{\rm max}).
\end{align}
Integrating \eqref{L1.1} over $(t,t+\tau)$, we obtain
\begin{align}\label{L1.3}
    \int_t^{t+\tau}\int_\Omega  \frac{n^2}{\log^\eta(n+e)}\,dx\,ds
    &\leq \frac{2}{\mu} \left(K_1\tau+\int_\Omega n(\cdot,t)\right) \notag\\
    &\leq C_3 \qquad \text{for all }t\in (0,T_{\rm max}-\tau),
\end{align}
where $C_3:=\frac{2}{\mu}(K_1+C_1)$.
If $T_{\max}=\infty$, choose $T_1>0$ such that
\[
    e^{-T_1}\int_\Omega n_0\leq 1.
\]
Then \eqref{L1.2} yields
\[
    \int_\Omega n(\cdot,t)\leq C_4:=1+K_1 \qquad\text{for all }t>T_1.
\]
Next we choose $T\geq T_1$ sufficiently large such that
\[
    e^{-\alpha(T-T_1)}\int_\Omega c(\cdot,T_1)\leq 1.
\]
For $t>T_1$, the bound for $n$ implies
\[
    \frac{d}{dt}\int_\Omega c+\alpha\int_\Omega c\leq \beta C_4.
\]
Hence Gronwall's inequality on $(T_1,t)$ yields
\[
    \int_\Omega c(\cdot,t) \leq C_5:=1+\frac{\beta C_4}{\alpha}
    \qquad\text{for all }t>T.
\]
Finally, since $\tau=1$ when $T_{\max}=\infty$, integrating \eqref{L1.1} over $(t,t+1)$ and using the bound for $n$, we get
\[
    \int_t^{t+1}\int_\Omega  \frac{n^2}{\log^\eta(n+e)}\,dx\,ds
    \leq C_6:=\frac{2}{\mu}(K_1+C_4) \qquad\text{for all }t>T.
\]
This implies \eqref{L1-3}, \eqref{L1-3'}, and \eqref{L1-4}.
\end{proof}

Thanks to the spatio-temporal information provided in the previous lemma and the Moser-Trudinger inequality as established in Lemma \ref{Moser-Trudinger}, we can now derive a uniform $L^2$ bound for $u$, as well as a spatio-temporal $L^2$ estimate for $\nabla u$. The main argument of the proof is adapted from Lemma 4.7 and Lemma 4.8 in \cite{Winkler_SIAM}. 

\begin{lemma} \label{L2u}
 There exist positive constants $C_1$, $C_2$, and $C_3$ such that
    \begin{align} \label{L2u-1}
        \int_\Omega |u(\cdot,t)|^2 \leq C_1 \qquad \text{for all } t \in (0,T_{\rm max}),
    \end{align}
    and 
    \begin{align} \label{L2u-2}
         \int_t^{t+\tau } \int_\Omega |\nabla u(\cdot,s)|^2 \,ds \leq C_2\qquad \text{for all } t \in (0,T_{\rm max}-\tau ),
    \end{align}
    as well as
        \begin{align} \label{L2u-3}
         \int_t^{t+\tau } \int_\Omega | u(\cdot,s)|^4 \,ds \leq C_3\qquad \text{for all } t \in (0,T_{\rm max}-\tau ),
    \end{align}
    where $\tau = \min \left \{1, \frac{T_{\rm max}}{2} \right \}$. Moreover, if $T_{\rm max} = \infty$, then there exist $T>0$ and positive constants $C_4$, $C_5$, and $C_6$ independent of initial data such that
     \begin{align} \label{L2u-4}
        \int_\Omega |u(\cdot,t)|^2 \leq C_4 \qquad \text{for all } t > T,
    \end{align}
    and 
    \begin{align} \label{L2u-5}
         \int_t^{t+1 } \int_\Omega |\nabla u(\cdot,s)|^2 \,ds \leq C_5 \qquad \text{for all } t > T,
    \end{align}
    as well as
      \begin{align} \label{L2u-6}
         \int_t^{t+1 } \int_\Omega | u(\cdot,s)|^4 \,ds \leq C_6 \qquad \text{for all } t > T.
    \end{align}
\end{lemma}
\begin{proof}
    The proofs of \eqref{L2u-1}, \eqref{L2u-2}, and \eqref{L2u-3} follow by arguments analogous to those in Lemma 3.4 in \cite{minhJMFM}; we therefore omit the details. We now assume that $T_{\rm max} = \infty$ and proceed to establish \eqref{L2u-4}, \eqref{L2u-5}, and \eqref{L2u-6}. By Poincaré's inequality, there exists $C_1$ depending only on $\Omega$ such that 
    \begin{align} \label{L2u.1}
        \int_\Omega |\nabla u|^2 \geq C_1 \int_\Omega |u|^2 \qquad \text{for all }t\in (0,\infty).
    \end{align}
  Let
\[
K_\phi:=2\|\nabla\phi\|_{L^\infty(\Omega)},\qquad
K_f:=\frac{8}{C_1}|\Omega|\,\|f\|_{L^\infty(\Omega\times(0,\infty))}^2 .
\]
Testing the third equation of \eqref{1} against $u=(u_1,u_2)$ and using Young's inequality, we obtain
\begin{align}\label{L2u.2}
    \frac{d}{dt}\int_\Omega |u|^2+2\int_\Omega |\nabla u|^2
    &\leq K_\phi \int_\Omega n |u_1| +K_\phi \int_\Omega n |u_2|
    +\frac{C_1}{8}\int_\Omega |u|^2+K_f 
\end{align}
for all $t\in (0, \infty)$. From \eqref{L1-3}, we can find $m_0>0$ independent of initial data and $T>0$ such that 
\[
\sup_{t\in(T,\infty)}\int_\Omega n(\cdot,t) \leq m_0,
\]
and 
\begin{align} \label{L2u.2'}
    \int_t^{t+1}\int_\Omega  \frac{n^2(\cdot,s)}{\log^\eta(n(\cdot,s)+e)}\,dx\,ds
    \leq m_0 \qquad  \text{for all }t>T .
\end{align}
 Moreover, by \eqref{L2u.1},
\[
\left(\int_\Omega |u_i|\right)^2 \leq |\Omega|\int_\Omega |u_i|^2
\leq \frac{|\Omega|}{C_1}\int_\Omega |\nabla u_i|^2 \qquad \text{for }i=1,2.
\]
Using this and Lemma \ref{Moser-Trudinger} with $\varepsilon=1$, we obtain
\begin{align}\label{L2u.3}
     K_\phi \int_\Omega n |u_1|+K_\phi\int_\Omega n |u_2|
     &\leq \frac{2K_\phi}{a}\int_\Omega n \ln \frac{n}{\bar n}
     +\left(\frac{K_\phi a m_0}{4\pi} +\frac{K_\phi M a m_0|\Omega|}{C_1}\right)
     \int_\Omega |\nabla u|^2  \notag\\
     &\quad+\frac{2K_\phi M}{a}\int_\Omega n  \qquad \text{for all }t>T.
\end{align}
Choose $a>0$ sufficiently small such that
\[
\frac{K_\phi a m_0}{4\pi} + \frac{K_\phi M a m_0|\Omega|}{C_1}
\leq \frac12 .
\]
Then
\begin{align}\label{L2u.4}
     K_\phi \int_\Omega n |u_1|+K_\phi\int_\Omega n |u_2|
     &\leq \frac{2K_\phi}{a}\int_\Omega n\ln n + \frac12\int_\Omega |\nabla u|^2+C_3 \qquad \text{for all }t>T,
\end{align}
where $C_3>0$ independent of initial data. Then there exists $C_4>0$ such that
\begin{align}\label{L2u.5}
    \frac{2K_\phi}{a}\int_\Omega n\ln n \leq \int_\Omega \frac{n^2}{\log^\eta(n+e)}+C_4 .
\end{align}
Collecting \eqref{L2u.1}, \eqref{L2u.2}, \eqref{L2u.4}, and \eqref{L2u.5}, we obtain
\begin{align}\label{L2u.6}
    \frac{d}{dt}\int_\Omega |u|^2 +\frac{C_1}{8}\int_\Omega |u|^2
    +\frac12\int_\Omega |\nabla u|^2
    \leq \int_\Omega \frac{n^2}{\log^\eta(n+e)}+C_5 \qquad \text{for all }t> T,
\end{align}
where $C_5:=C_3+C_4+K_f$. Using estimate \eqref{L2u.2'} and applying Lemma \ref{ODI} to \eqref{L2u.6}, we obtain \eqref{L2u-4}. Integrating \eqref{L2u.6} over $(t,t+1 )$ and using \eqref{L2u-4} leads to 
    \begin{align}\label{L2u.7}
        \frac{1}{2}\int_t^{t+1 }\int_\Omega |\nabla u|^2 \leq  \int_t^{t+1 }\int_\Omega \frac{n^2}{\ln^\eta (n+e)} +C_5  +\int_\Omega |u(\cdot,t)|^2 \leq C_6\quad\text{for all }t>T,
    \end{align}
    where $C_6>0$ independent of initial data, which implies \eqref{L2u-5}. By applying the Gagliardo–Nirenberg interpolation inequality and using \eqref{L2u-4} and \eqref{L2u-5}, we obtain that 
    \begin{align}\label{L2u.8}
        \int_t^{t+1 }\int_\Omega | u|^4 &\leq  C_7 \int_t^{t+1} \int_\Omega |u|^2 \cdot \int_\Omega |\nabla u|^2 +C_7 \int_t ^{t+1} \left ( \int_\Omega |u|^2 \right )^2 \notag \\
        &\leq C_8 \qquad \text{for all }t>T,
    \end{align}
    where $C_7>0$ and $C_8>0$ independent of initial data, which implies \eqref{L2u-6}. The proof is now complete.
\end{proof}

The following step is to establish an $L^2$ bound for $c$. 

\begin{lemma} \label{c2}
    There exist positive constants $C_1$ and $C_2$ such that 
    \begin{align} \label{c2-1}
        \int_\Omega c^2(\cdot,t) \leq C_1 \qquad \text{for all }t\in (0,T_{\rm max}),
    \end{align}
    and 
     \begin{align} \label{c2-1'}
   \int_t^{t+\tau } \int_\Omega |\nabla c(\cdot,s)|^2\,dx\,ds
   \leq C_2 \qquad \text{for all }t\in (0,T_{\rm max}-\tau),
\end{align}
where $\tau = \min \left \{ \frac{T_{\rm max}}{2},1 \right \}$. Moreover, if $T_{\rm max}=\infty$, then there exists $T>0$ such that
\begin{align} \label{c2-2}
    \int_\Omega c^2(\cdot,t) \leq C_3 \qquad \text{for all }t>T,
\end{align}
and
\begin{align} \label{c2-2'}
   \int_t^{t+1 } \int_\Omega |\nabla c(\cdot,s)|^2\,dx\,ds
   \leq C_4 \qquad \text{for all } t>T,
\end{align}
where $C_3$ and $C_4$ are positive constants independent of the initial data.
\end{lemma}
\begin{proof}
    By applying the Gagliardo–Nirenberg interpolation inequality and using the estimate \eqref{L1-1'}, we obtain  
    \begin{align} \label{c2.1}
        \int_\Omega c^3 &\leq C_1 \int_\Omega c \cdot \int_\Omega |\nabla c|^2 +C_1 \left ( \int_\Omega c \right )^3 \notag \\
        &\leq C_2  \int_\Omega |\nabla c|^2 +C_3,
    \end{align}
    where $C_1,C_2,$ and $C_3$ are positive constants. Multiplying the second equation of \eqref{1} by $c$, applying Young's inequality and using \eqref{c2.1} yields
   \begin{align} \label{c2.2}
    \frac{1}{2} \frac{d}{dt} \int_\Omega c^2
    + \alpha \int_\Omega c^2  + \int_\Omega |\nabla c|^2
    &= \beta\int_\Omega nc \notag \\
    &\leq \frac{1}{2C_2} \int_\Omega c^3 + C_4 \int_\Omega n^{\frac{3}{2}} \notag \\
    &\leq  \frac{1}{2} \int_\Omega |\nabla c|^2
    + C_4 \int_\Omega \frac{n^2}{\log^\eta(n+e)} +C_5,
\end{align}
    where $C_4>0$ and $C_5>0$. This leads to 
    \begin{align}\label{c2.3}
          \frac{1}{2} \frac{d}{dt} \int_\Omega c^2 + \alpha \int_\Omega c^2 + \frac{1}{2}\int_\Omega |\nabla c|^2 \leq  C_4 \int_\Omega \frac{n^2}{\log^\eta (n+e)} +C_5 \qquad \text{for all }t\in (0,T_{\rm max}).
    \end{align}
    Thanks to \eqref{L1-2}, we apply Lemma \ref{ODI} to \eqref{c2.3} to deduce \eqref{c2-1}. 
Integrating \eqref{c2.3} over $(t,t+\tau)$ and using \eqref{L1-2} and \eqref{c2-1}, we get
\begin{align}
    \frac{1}{2} \int_t^{t+\tau }\int_\Omega |\nabla c(\cdot,s)|^2\,dx\,ds
    &\leq C_4 \int_t ^{t+\tau }\int_\Omega  \frac{n^2}{\log^\eta(n+e)}\,dx\,ds
    +C_5 \tau + \frac{1}{2}\int_\Omega c^2(\cdot,t) \notag\\
    &\leq C_6 \qquad \text{for all }t\in (0,T_{\rm max}-\tau),
\end{align}
  where $C_6>0$, which implies \eqref{c2-1'}. Now assume that $T_{\rm max} = \infty$. To establish \eqref{c2-2} and \eqref{c2-2'}, it is sufficient to replicate the preceding argument, substituting the estimate \eqref{L1-3'} for \eqref{L1-1'} and the estimate \eqref{L1-4} for \eqref{L1-2}.
\end{proof}
A major difficulty in deriving an $L^2$ bound for $\nabla c$ directly from the sub-logistic damping arises from the fact that the boundedness of $n$ in $L^2((0,T); L^2(\Omega))$ for any $T<T_{\rm max}$ is not known. To overcome this issue, we aim to decouple $\int_\Omega |\nabla c|^2$ from $\int_\Omega n\log n$ so as to exploit the sub-logistic source term effectively. The following lemma provides a crucial estimate that enables us to handle the term $\int_\Omega |\nabla c|^2$.

\begin{lemma} \label{gradL2c}
    There exists $C>0$ independent of initial data such that 
    \begin{align}\label{gradL2c-1}
        \frac{d}{dt}\int_\Omega |\nabla c|^2 + \int_\Omega |\Delta c|^2 \leq 2\beta^2 \int_\Omega n^2+\left (C \int_\Omega |u|^4+1 \right ) \cdot \int_\Omega |\nabla c|^2 \qquad \text{for all }t\in (0,T_{\rm max}).
    \end{align}
\end{lemma}
\begin{proof}
     From the second equation of \eqref{1}, integrating by parts, we have 
    \begin{align}\label{gradL2c.1'}
       \frac{1}{2} \frac{d}{dt} \int_\Omega |\nabla c|^2 + \int_\Omega |\Delta c|^2 + \alpha \int_\Omega  |\nabla c|^2 = - \beta  \int_\Omega n \Delta c + \int_\Omega (u \cdot \nabla c)\Delta c \qquad \text{for all }t\in (0,T_{\rm max}).
    \end{align}
By Young's and Hölder's inequalities,
    \begin{align}\label{gradL2c.2'}
        -\beta \int_\Omega n \Delta c \leq \frac{1}{4}\int_\Omega |\Delta c|^2 + \beta^2 \int_\Omega n^2,
    \end{align}
    and 
    \begin{align}\label{gradL2c.1}
    \int_\Omega (u \cdot \nabla c)\Delta c \leq \left \| u \right \|_{L^4(\Omega)} \left \| \nabla c \right \|_{L^4(\Omega)}  \left \| \Delta c \right \|_{L^2(\Omega)}.
    \end{align}
Thanks to the Gagliardo–Nirenberg interpolation inequality, 
    \begin{align}\label{gradL2c.2}
        \left \| \nabla c \right \|_{L^4(\Omega)} \leq C_1 \left \| \nabla c \right \|^{\frac{1}{2}}_{L^2(\Omega)} \left \| \Delta c \right \|^{\frac{1}{2}}_{L^2(\Omega)} +C_1   \left \| \nabla c \right \|_{L^2(\Omega)} ,
    \end{align}
    where $C_1>0$. Combining \eqref{gradL2c.1} and \eqref{gradL2c.2}, and then applying Young's inequality, we infer that
\begin{align}\label{gradL2c.3}
\int_\Omega (u \cdot \nabla c)\Delta c
&\leq C_1\|u\|_{L^4(\Omega)}\|\nabla c\|_{L^2(\Omega)}^{\frac12}
\|\Delta c\|_{L^2(\Omega)}^{\frac32}
+C_1\|u\|_{L^4(\Omega)}\|\nabla c\|_{L^2(\Omega)}
\|\Delta c\|_{L^2(\Omega)} \notag\\
&\leq \frac14\|\Delta c\|_{L^2(\Omega)}^2
+\left(C_2\|u\|_{L^4(\Omega)}^4+\frac12\right)
\|\nabla c\|_{L^2(\Omega)}^2,
\end{align}
where $C_2>0$ depends only on $\Omega$.
Combining \eqref{gradL2c.1'}, \eqref{gradL2c.2'} and \eqref{gradL2c.3}, we obtain \eqref{gradL2c-1}.
\end{proof}

Next, we derive an elementary estimate which will later be applied in establishing an $L\log L$ bound for $n$ and $L^2$ bound for $\nabla c$ in Lemma \ref{LlogL-p}.

\begin{lemma} \label{Log}
    There exists $C>0$ such that 
    \begin{equation}\label{Log-1}
        \int_\Omega (\log n+1)\left ( rn - \frac{\mu n^2}{\log^\eta(n+e)} \right ) \leq - \frac{5\mu}{6} \int_\Omega n^2 \log^{1-\eta }(n+e)+C \qquad \text{for all }t\in (0,T_{\rm max}).
    \end{equation}
\end{lemma}
\begin{proof}
    Setting 
    \begin{equation*}
        g(s)= rs(\log s+1) + \frac{5 \mu }{6} s^2 \log^{1-\eta }(s+e) \qquad \text{for all }s>0
    \end{equation*}
    and 
    \begin{equation*}
        h(s) = \frac{5\mu }{6} \frac{s^2(\log s+1)}{\log^\eta (s+e)}\qquad \text{for all }s>0,
    \end{equation*}
    then we find that 
    \[
    \lim_{s \to \infty} \frac{g(s)}{h(s)} =1.
    \]
    This implies that there exists $N>e$ such that 
    \begin{align} \label{Log.1}
        g(s) \leq \frac{6}{5}h(s)\qquad \text{for all }s> N.
    \end{align}
    Noting that there exists a positive constants $C_1$ such that $h(s) \geq -C_1$ for all $s>0$, we infer that 
    \begin{align*}
        -\frac{6}{5}\int_{\left\{ n \leq N \right \}}h(n) \leq \frac{6}{5}|\Omega| C_1.
    \end{align*}
    This, together with \eqref{Log.1} leads to 
    \begin{align*}
        \int_{\left\{ n > N \right \}}g(n) &\leq \frac{6}{5} \int_{\left\{ n > N \right \}}h(n) \notag \\
        &\leq \frac{6}{5}\int_{\Omega}h(n)-\frac{6}{5} \int_{\left\{ n \leq  N \right \}}h(n) \notag \\
        &\leq \frac{6}{5}\int_{\Omega}h(n) +\frac{6}{5}|\Omega|C_1.
    \end{align*}
  Since $s\log s\to0$ as $s \to 0$, $g$ has a continuous extension to $[0,N]$.
Take
\[
C_N:=\max\left\{0,\sup_{0\leq s\leq N}g(s)\right\}.
\]
Then
\begin{align}\label{Log.2}
    \int_\Omega g(n) 
    &= \int_{\{n>N\}} g(n)+\int_{\{n\leq N\}} g(n) \notag\\
    &\leq \frac{6}{5}\int_{\Omega}h(n) + \frac{6}{5}|\Omega|C_1+C_N|\Omega|.
\end{align}
   Hence, we have
     \begin{equation*}
        r\int_\Omega n (\log n +1) + \frac{5\mu }{6}\int_\Omega n^2 \log^{1-\eta}(n+e) \leq \mu \int_\Omega \frac{n^2(\log n+1)}{\log^{\eta }(n+e)} + C_2, 
    \end{equation*}
    where $C_2= \frac{6|\Omega|C_1}{5}+C_N|\Omega|$. This implies \eqref{Log-1} and thereby concludes the proof.
\end{proof}

The following lemma plays a central role in handling the singular behavior of $ c $ near zero.

\begin{lemma} \label{nlogc}
The following estimate holds:
\begin{align*}
     -\frac{d}{dt} \int_\Omega n \log c &\leq 2 \int_\Omega \frac{|\nabla n|^2}{n} - \frac{1}{2}\int_\Omega n \frac{|\nabla c|^2}{c^2} -r \int_\Omega n \log c \\
     &\quad+ \mu  \int_\Omega \frac{n^2\log c}{\log^\eta(n+e)}  +\alpha \int_\Omega n - \beta \int_\Omega \frac{n^2}{c}
\end{align*}
    for all $t\in (0,T_{\rm max})$.
\end{lemma}

\begin{proof}
    Direct calculation shows that 
    \begin{align} \label{nlogc.1}
        -\frac{d}{dt} \int_\Omega n \log c = - \int_\Omega n_t \log c - \int_\Omega \frac{n}{c}c_t =I+J.
        \end{align}
    Using the first equation and integrating by parts gives us that 
    \begin{align} \label{nlogc.2}
        I&= -\int_\Omega \log c \left ( -u \cdot \nabla n + \Delta n -\chi \nabla \cdot \left ( n \frac{\nabla c}{c^k} \right )+rn - \frac{\mu n^2}{\log^\eta(n+e)}  \right ) \notag \\
        &= \int_\Omega    u \cdot \nabla n  \log c+ \int_\Omega \frac{\nabla n \cdot \nabla c}{c} -\chi \int_\Omega n \frac{|\nabla c|^2}{c^{1+k}} -r \int_\Omega n \log c + \mu  \int_\Omega \frac{ n^2 \log c}{\log^\eta(n+e)}.
    \end{align}
    Integrating by parts and then using the fact that $\nabla \cdot u =0$ and $u=0$ on $\partial \Omega$, we deduce that 
    \begin{align} \label{nlogc.3}
        \int_\Omega    u \cdot \nabla n  \log c &= - \int_\Omega n \frac{u \cdot \nabla c}{c} .
    \end{align}
    Similarly, we handle $J$ as follows:
    \begin{align} \label{nlogc.4}
        J&= - \int_\Omega \frac{n}{c} \left (-u \cdot \nabla c + \Delta c- \alpha c+ \beta n \right ) \notag \\
        &= \int_\Omega n \frac{u \cdot \nabla c}{c} +\int_\Omega \frac{\nabla n \cdot \nabla c}{c} - \int_\Omega n \frac{|\nabla c|^2}{c^2} +\alpha \int_\Omega n - \beta \int_\Omega \frac{n^2}{c}.
    \end{align}
    Collecting from \eqref{nlogc.1} to \eqref{nlogc.4}, and using Young's inequality, we arrive at 
    \begin{align*}
         -\frac{d}{dt} \int_\Omega n \log c &= 2 \int_\Omega \frac{\nabla n \cdot \nabla c}{c} - \int_\Omega n \frac{|\nabla c|^2}{c^2}-\chi \int_\Omega n \frac{|\nabla c|^2}{c^{1+k}} -r \int_\Omega n \log c  \notag \\
         &\quad+ \mu  \int_\Omega \frac{ n^2 \log c}{\log^\eta(n+e)}+\alpha \int_\Omega n - \beta \int_\Omega \frac{n^2}{c} \notag \\
         &\leq 2 \int_\Omega \frac{|\nabla n|^2}{n} - \frac{1}{2}\int_\Omega n \frac{|\nabla c|^2}{c^2} -r \int_\Omega n \log c + \mu  \int_\Omega \frac{n^2\log c}{\log^\eta(n+e)}+\alpha \int_\Omega n - \beta \int_\Omega \frac{n^2}{c}, 
    \end{align*}
    for all $t\in (0,T_{\rm max})$, which completes the proof.
\end{proof}

To prepare for Lemma \ref{LlogL-p}, we derive the following estimate allowing us to deal with chemo-attractant term.

\begin{lemma} \label{LlolL'}
    There exists $C>0$ independent of initial data such that 
    \begin{align}
          \chi \int_\Omega \nabla c \cdot \frac{\nabla n}{c^k} &\leq   \frac{1}{3} \int_\Omega \frac{|\nabla n|^2}{n} + \frac{1}{12} \int_\Omega n \frac{|\nabla c|^2}{c^2}+\frac{1}{2} \int_\Omega |\Delta c|^2 \notag \\
          &\quad+ \frac{\beta}{6}\int_\Omega \frac{n^2}{c}+ C \int_\Omega n^2 +\frac{1}{2}\int_\Omega c^2 \quad \text{for all }t\in (0,T_{\rm max}). 
    \end{align}
\end{lemma}
\begin{proof}
In case $k \in (0, \frac{1}{2})$, by integration by parts, we obtain that 
\begin{align}\label{LlolL-p.2}
    \chi \int_\Omega \nabla c \cdot \frac{\nabla n}{c^k} &= -\chi \int_\Omega \frac{n}{c^k} \Delta c +\chi k \int_\Omega n \frac{|\nabla c|^2}{c^{1+k}}.
\end{align}
Applying Young's inequality, we have 
\begin{align}\label{LlolL-p.3}
    -\chi \int_\Omega \frac{n}{c^k} \Delta c &\leq \frac{1}{4}\int_\Omega |\Delta c|^2 +\chi^2 \int_\Omega \frac{n^2}{c^{2k}} \notag \\
    &\leq  \frac{1}{4}\int_\Omega |\Delta c|^2 + \frac{\beta }{6}\int_\Omega \frac{n^2}{c} +C_1 \int_\Omega n^2,
\end{align}
for some $C_1>0$. By Lemma \ref{LK-1}, we can find $C_2>0$ such that 
\begin{align}\label{LlolL-p.4}
    \int_\Omega \frac{|\nabla c|^4}{c^2} \leq C_2 \int_\Omega |\Delta c|^2+ C_2 \int_\Omega c^2.
\end{align}
Using this and employing Young's inequality, we obtain 
\begin{align}\label{LlolL-p.5}
    \chi k \int_\Omega n \frac{|\nabla c|^2}{c^{1+k}} &\leq \frac{1}{12}\int_\Omega n\frac{|\nabla c|^2}{c^{2}} + C_3 \int_\Omega n \frac{|\nabla c|^2}{c} \notag \\
    &\leq  \frac{1}{12}\int_\Omega n\frac{|\nabla c|^2}{c^{2}} +\frac{1}{4C_2}\int_\Omega \frac{|\nabla c|^4}{c^2} + C_4 \int_\Omega n^2 \notag \\
    &\leq  \frac{1}{12}\int_\Omega n\frac{|\nabla c|^2}{c^{2}} +\frac{1}{4}\int_\Omega |\Delta c|^2 + \frac{1}{4}\int_\Omega c^2 +C_4 \int_\Omega n^2.
\end{align}
Collecting \eqref{LlolL-p.2}, \eqref{LlolL-p.3}, and \eqref{LlolL-p.5}, we infer that 
\begin{align}\label{LlolL-p.4'}
      \chi \int_\Omega \nabla c \cdot \frac{\nabla n}{c^k}  \leq \frac{1}{12}\int_\Omega n \frac{|\nabla c|^2}{c^2}+\frac{1}{2}\int_\Omega |\Delta c|^2 + \frac{\beta}{6}\int_\Omega \frac{n^2}{c}  +C_5\int_\Omega n^2+\frac{1}{4}\int_\Omega c^2,
\end{align}
where $C_5=C_1+C_4$. In case $k \in [\frac{1}{2},1)$, we apply Young's inequality and use \eqref{LlolL-p.4} to obtain that 
\begin{align} \label{LlolL-p.6}
     \chi \int_\Omega \nabla c \cdot \frac{\nabla n}{c^k} &\leq \frac{1}{3} \int_\Omega \frac{|\nabla n|^2}{n} + \frac{3\chi^2}{4} \int_\Omega n\frac{|\nabla c|^2}{c^{2k}} \notag \\
     &\leq \frac{1}{3} \int_\Omega \frac{|\nabla n|^2}{n} + \frac{1}{12} \int_\Omega n \frac{|\nabla c|^2}{c^2}+C_6 \int_\Omega n \frac{|\nabla c|^2}{c} \notag \\
     &\leq \frac{1}{3} \int_\Omega \frac{|\nabla n|^2}{n} + \frac{1}{12} \int_\Omega n \frac{|\nabla c|^2}{c^2}+\frac{1}{2C_2} \int_\Omega \frac{|\nabla c|^4}{c^2} + C_7 \int_\Omega n^2 \notag \\
     &\leq \frac{1}{3} \int_\Omega \frac{|\nabla n|^2}{n} + \frac{1}{12} \int_\Omega n \frac{|\nabla c|^2}{c^2}+\frac{1}{2} \int_\Omega |\Delta c|^2 + \frac{1}{2} \int_\Omega c^2 + C_7 \int_\Omega n^2 ,
\end{align}
where $C_6$ and $C_7$ are positive constants. Finally, combining \eqref{LlolL-p.4'} and \eqref{LlolL-p.6}, we obtain the desired result.
\end{proof}

We next derive a simple auxiliary estimate that will be used in the proof of Lemma \ref{LlogL-p}.

\begin{lemma} \label{Lnlogc}
    There exists $C_1>0$ such that 
    \begin{align} \label{Lnlogc-1}
        - \int_t ^{t+\tau }\int_\Omega n(\cdot,s) \log c(\cdot,s) \, ds \leq C_1 \qquad \text{for all }t \in (0,T_{\rm max}-\tau),
    \end{align}
    where $\tau = \min \left \{ \frac{T_{\rm max}}{2},1 \right \}$. Furthermore, if $T_{\rm max}= \infty$ then there exists $T>0$ such that
     \begin{align}\label{Lnlogc-2}
        - \int_t ^{t+1 }\int_\Omega n(\cdot,s) \log c(\cdot,s) \, ds \leq C_2 \qquad \text{for all }t >T,
    \end{align}
    where $C_2>0$ independent of initial data.
\end{lemma}

\begin{proof}
    Testing the second equation of \eqref{1} by $\log c+1$ and using the fourth equation of \eqref{1} and boundary condition of $u$  and then employing the estimate \eqref{L1-1}  yields that 
    \begin{align} \label{Lnlogc.1}
        \frac{d}{dt}\int_\Omega c\log c &= \int_\Omega (\log c+1) \left ( - u \cdot \nabla c+ \Delta c-\alpha c+ \beta n \right ) \notag \\
        &= - \int_\Omega \frac{|\nabla c|^2}{c} - \alpha \int_\Omega c \log c- \alpha \int_\Omega c + \beta \int_\Omega n \log c + \beta \int_\Omega n \notag \\
        &\leq \frac{\alpha |\Omega|}{e}+  \beta \int_\Omega n \log c +C_1 \qquad \text{for all }t\in (0,T_{\rm max}),
    \end{align}
    where we used $x\log x \geq  -\frac{1}{e}$ for all $x>0$ and $C_1=  \beta \sup_{t\in (0,T_{\rm max})} \int_\Omega n(\cdot,s)$. This, together with the estimate \eqref{c2-1}, implies that
    \begin{align} \label{Lnlogc.2}
        -\beta \int_t^{t+\tau }  \int_\Omega n(\cdot,s) \log c(\cdot,s)\, ds  &\leq  \left ( \frac{\alpha|\Omega|}{e}+C_1 \right )\tau  + \int_\Omega c(\cdot,t)\log c(\cdot,t)-\int_\Omega c(\cdot,t+\tau)\log c(\cdot,t+\tau ) \notag \\
        &\leq \left ( \frac{\alpha|\Omega|}{e}+C_1 \right )\tau + \int_\Omega c^2(\cdot,t) + \frac{|\Omega|}{e} \notag \\
        &\leq C_2 \qquad \text{for all }t\in (0,T_{\rm max}-\tau ),
    \end{align}
    where $C_2=\left ( \frac{\alpha|\Omega|}{e}+C_1 \right )\tau  + \frac{|\Omega|}{e} + \sup_{t\in (0,T_{\rm max})} \int_\Omega c^2(\cdot,t)$, which implies \eqref{Lnlogc-1}. Now, under the assumption $T_{\max} = \infty$, the proof of \eqref{Lnlogc-2} follows the same reasoning as \eqref{Lnlogc.1} and \eqref{Lnlogc.2} except for using \eqref{L1-3} and \eqref{c2-2} in place of \eqref{L1-1} and \eqref{c2-1}, respectively.

\end{proof}

\section{\texorpdfstring{ Boundedness of $n$ in $L\log L$}{L log L and L^p Boundedness of Solutions}} \label{S4}
In this section, we first establish a uniform $L \log L$ bound for $n$, which subsequently enables us to derive $L^p$ boundedness in Section~\ref{S4'}. We then prove the uniform boundedness of $u$ in $W^{1,2}(\Omega)$. 

We start with the most delicate part, the derivation of a uniform $L \log L$ bound for $n$.

\begin{lemma} \label{LlogL-p}
There exists a positive constant $C_1$ such that
    \begin{align}\label{LlogL-p-1}
        \int_\Omega n(\cdot,t) \log n(\cdot,t) \leq C_1 \qquad \text{for all } t \in (0,T_{\rm max}),
    \end{align}
    and 
    \begin{align}\label{LlogL-p-1'}
        \int_\Omega |\nabla c(\cdot,t) |^2 \leq C_1 \qquad \text{for all } t \in (0,T_{\rm max}),
    \end{align}
    as well as 
    \begin{align}\label{LlogL-p-1''}
        \int_t^{t+\tau}\int_\Omega n^2(\cdot,s) \log^{1-\eta}(n(\cdot,s)+e) \,ds \leq C_1 \qquad \text{for all } t \in (0,T_{\rm max} -\tau).
    \end{align}
    
    Furthermore, if $T_{\rm max}= \infty$ then there exists $T>0$ such that 
    \begin{align}\label{LlogL-p-2}
           \int_\Omega n(\cdot,t) \log n(\cdot,t) \leq C_2 \qquad \text{for all } t >T,
    \end{align}
      and 
    \begin{align}\label{LlogL-p-2'}
        \int_\Omega |\nabla c(\cdot,t) |^2 \leq C_2 \qquad \text{for all } t >T,
    \end{align}
      as well as 
    \begin{align}\label{LlogL-p-2''}
        \int_t^{t+1}\int_\Omega n^2(\cdot,s) \log^{1-\eta}(n(\cdot,s)+e) \,ds \leq C_2 \qquad \text{for all } t>T,
    \end{align}
    where $C_2>0$ is independent of initial data.
\end{lemma}
\begin{proof}
    Testing the first equation \eqref{1} by $\log n +1$ and using the fact that $\nabla \cdot u =0$ and $u=0$ on $\partial \Omega$ yields
\begin{align} \label{M.1}
    \frac{d}{dt}\int_\Omega n \log n &= \int_\Omega (\log n+1) \left ( -u \cdot \nabla n + \Delta n -\chi \nabla \cdot \left ( \frac{n}{c^k} \nabla c \right ) +rn - \frac{\mu n^2}{\log^\eta(n+e)}\right ) \notag\\
    &= -\int_\Omega \frac{|\nabla n|^2}{n} + \chi \int_\Omega \nabla n \cdot \frac{\nabla c}{c^k} + \int_\Omega (\log n+1)\left ( rn - \frac{\mu n^2}{\log^\eta(n+e)} \right ),
\end{align}
for all $t\in (0,T_{\rm max})$. Thanks to Lemma \ref{LlolL'}, we can find $C_1>0$ independent of initial data such that 
  \begin{align} \label{M.2}
          \chi \int_\Omega \nabla c \cdot \frac{\nabla n}{c^k} &\leq   \frac{1}{3} \int_\Omega \frac{|\nabla n|^2}{n} + \frac{1}{12} \int_\Omega n \frac{|\nabla c|^2}{c^2}+\frac{1}{2} \int_\Omega |\Delta c|^2 \notag \\
          &\quad+ \frac{\beta}{6}\int_\Omega \frac{n^2}{c}+ C_1 \int_\Omega n^2 +\frac{1}{2}\int_\Omega c^2 \qquad \text{for all }t\in (0,T_{\rm max}). 
    \end{align}
Applying Lemma \ref{nlogc}, we obtain that 
\begin{align} \label{M.3}
     -\frac{1}{3}\frac{d}{dt} \int_\Omega n \log c &\leq \frac{2}{3} \int_\Omega \frac{|\nabla n|^2}{n} - \frac{1}{6}\int_\Omega n \frac{|\nabla c|^2}{c^2} - \frac{r}{3} \int_\Omega n \log c  \notag \\
     &\quad+ \frac{\mu}{3}   \int_\Omega \frac{n^2\log c}{\log^\eta(n+e)} \quad+\frac{\alpha}{3} \int_\Omega n - \frac{\beta }{3} \int_\Omega \frac{n^2}{c}\qquad \text{for all }t\in (0,T_{\rm max}).
\end{align}
By Lemma \ref{gradL2c}, we can find $C_2>0$ independent of initial data such that
\begin{align} \label{M.4}
    \frac{1}{2}\frac{d}{dt} \int_\Omega |\nabla c|^2
    + \frac{1}{2}\int_\Omega |\Delta c|^2
    &\leq C_2\left(1+\int_\Omega |u|^4\right)
    \int_\Omega |\nabla c|^2  +  \beta^2 \int_\Omega n^2
\end{align}
for all $t\in (0,T_{\rm max})$. From Lemma \ref{Log}, we can find $C_3>0$ independent of initial data such that 
\begin{align}\label{M.5}
      \int_\Omega (\log n+1)\left ( rn - \frac{\mu n^2}{\log^\eta(n+e)} \right ) \leq - \frac{5\mu}{6} \int_\Omega n^2 \log^{1-\eta }(n+e)+C_3\qquad \text{for all }t\in (0,T_{\rm max}).
\end{align}
Collecting from \eqref{M.1} to \eqref{M.5}, we arrive at 
\begin{align}\label{M.6}
   \frac{d}{dt} \left \{\int_\Omega n \log n - \frac{1}{3}\int_\Omega n \log c
   + \frac{1}{2} \int_\Omega |\nabla c|^2 \right \} 
   &\leq C_2\left(1+\int_\Omega |u|^4\right)\int_\Omega |\nabla c|^2
   - \frac{r}{3}\int_\Omega n \log c \notag \\
   &\quad+ \frac{1}{2}\int_\Omega c^2 +C_4 \int_\Omega n^2
   + \frac{\alpha}{3}\int_\Omega n
   - \frac{\beta}{6}\int_\Omega \frac{n^2}{c}  \notag \\
   &\quad  + \frac{\mu}{3} \int_\Omega \frac{n^2 \log c}{\log^\eta (n+e)}
   - \frac{5\mu}{6} \int_\Omega n^2 \log^{1-\eta }(n+e)+C_3 
\end{align}
for all $t\in (0,T_{\rm max})$. Thanks to Young's inequality that 
\begin{align}\label{M.6'}
    ab \leq a\ln a +e^{b-1} \qquad \text{for all }a>0 \text{ and } b\in \mathbb{R},
\end{align}
and \eqref{L1-1'}, we infer that
\begin{align}\label{M.7}
     \frac{\mu}{3} \int_\Omega \frac{n^2 \log c}{\log^\eta (n+e)} &\leq \frac{\mu}{3} \int_\Omega \frac{n^2}{\log^\eta (n+e)} \log \left (  \frac{n^2}{\log^\eta (n+e)} \right ) + \frac{\mu}{3e}\int_\Omega c \notag \\
     &\leq \frac{2\mu}{3} \int_\Omega n^2 \log^{1-\eta }(n+e) + C_5,
\end{align}
where $C_5>0$.
Since $-\log x \leq x^{-1/2}$ for all $x>0$, it follows that 
\begin{align}\label{M.8}
    -\frac{r}{3}\int_\Omega n \log c 
    &\leq \frac{r}{3}\int_\Omega \frac{n}{\sqrt{c}} \notag \\
    &\leq \frac{\beta}{6}\int_\Omega \frac{n^2}{c}+ C_6,
\end{align}
for some $C_6>0$. Since $\eta \in (0,1)$, we can find $C_7>0$ such that
\begin{align}\label{M.9}
   C_4 \int_\Omega n^2 \leq \frac{\mu}{12} \int_\Omega n^2 \log^{1-\eta}(n+e)+C_7.
\end{align}
Applying \eqref{M.6'}, \eqref{L1-1'}, and using the fact that $x\log x \geq -\frac{1}{e}$ for all $x>0$, we find that
\begin{align}\label{M.10}
    -\int_\Omega n \log n +\frac{1}{3} \int_\Omega n \log c &\leq -\frac{2}{3} \int_\Omega n \log n + \frac{1}{3e}\int_\Omega c \notag \\
    &\leq C_8 \qquad \text{for all } t\in (0,T_{\rm max}).
\end{align}
Setting 
\begin{equation*}
    y(t) := \int_\Omega n \log n -\frac{1}{3} \int_\Omega n \log c + \frac{1}{2} \int_\Omega |\nabla c|^2 +C_8 \qquad \text{for all }t\in (0,T_{\rm max}),
\end{equation*}
then \eqref{M.10} warrants that $y$ is nonnegative.  Collecting from 
\eqref{M.6} to \eqref{M.10} and using the estimates \eqref{L1-1} and \eqref{c2-1} yields 
\begin{align}\label{M.11}
    y'(t) +\frac{\mu}{12}\int_\Omega n^2 \log^{1-\eta}(n+e)
    \leq C_2\left(1+\int_\Omega |u(\cdot,t)|^4\right)y(t)+C_{9}
\end{align}
for all $t\in (0,T_{\rm max})$, where $C_{9}>0$. From the estimates 
\eqref{L1-2} and \eqref{c2-1'} and Lemma \ref{Lnlogc}, there exists $C_{10}>0$ such that 
\begin{align} \label{M.12}
    \int_t^{t+\tau }y(s)\, ds \leq C_{10} 
    \qquad \text{for all }t \in (0, T_{\rm max} -\tau ).
\end{align}
Setting 
\[
h(t)= C_2\left(1+\int_\Omega |u(\cdot,t)|^4\right) \qquad\text{and}\qquad 
g(t)=C_{9},
\]
it follows from \eqref{L2u-3} that there exists $C_{11}>0$ such that
\begin{align}\label{M.13}
    \int_t^{t+\tau }h(s)\, ds \leq C_{11} 
    \qquad \text{for all }t \in (0, T_{\rm max} -\tau ),
\end{align}
and
\begin{align}\label{M.14}
    \int_t^{t+\tau }g(s)\, ds \leq C_{9} \qquad \text{for all }t \in (0, T_{\rm max} -\tau ).
\end{align}
Combining \eqref{M.11}--\eqref{M.14} and applying Lemma \ref{ODI'}, we obtain
\[
    y(t)\leq C_{12}
    \qquad\text{for all }t\in(\tau,T_{\rm max}).
\]
By local regularity and the assumptions on the initial data,
\[
    \sup_{0<t\leq\tau}y(t)<\infty.
\]
Hence, after enlarging $C_{12}$ if necessary,
\begin{align}\label{M.14'}
    y(t)\leq C_{12} \qquad\text{for all }t\in(0,T_{\rm max}).
\end{align}
This, together with \eqref{M.10}, implies that

\begin{align*}
    \int_\Omega |\nabla c(\cdot,t)|^2 &\leq  2C_{12}-2\int_\Omega n \log n +\frac{2}{3} \int_\Omega n \log c \notag \\
    &\leq C_{13}\qquad \text{for all }t \in (0, T_{\rm max} ),
\end{align*}
where $C_{13}>0$. By applying \eqref{M.6'} and \eqref{L1-1'}, we can find $C_{14}>0$ such that 
\begin{align*}
    \int_\Omega n \log n &\leq C_{12} + \frac{1}{3}\int_\Omega n \log c \notag \\  
    &\leq C_{12}+ \frac{1}{3}\int_\Omega n \log n + \frac{1}{3e}\int_\Omega c \notag \\
    &\leq C_{14}+ \frac{1}{3}\int_\Omega n \log n,
\end{align*}
which further entails that
\begin{align}\label{M.15}
    \int_\Omega n(\cdot,t) \log n(\cdot,t) \leq \frac{3C_{14}}{2}\qquad \text{for all }t \in (0, T_{\rm max} ).
\end{align}
Integrating \eqref{M.11} over $(t,t+\tau)$ and using non-negativity of $y$, 
\eqref{M.13}, and \eqref{M.14'}, we obtain
\begin{align}\label{M.17}
    \frac{\mu}{12}\int_t^{t+\tau}\int_\Omega 
    n^2\log^{1-\eta}(n+e)\,dx\,ds
    &\leq y(t)-y(t+\tau)+\int_t^{t+\tau}h(s)y(s)\,ds+C_9\tau \notag\\
    &\leq C_{15}
    \qquad \text{for all }t \in (0, T_{\rm max} -\tau),
\end{align}
for some $C_{15}>0$.
Under the assumption $T_{\rm max}=\infty$, we repeat the argument above on a sufficiently large interval, using 
\eqref{L1-3}, \eqref{L1-3'}, \eqref{L1-4}, \eqref{c2-2}, \eqref{c2-2'}, 
\eqref{L2u-6}, and \eqref{Lnlogc-2} in place of 
\eqref{L1-1}, \eqref{L1-1'}, \eqref{L1-2}, \eqref{c2-1}, \eqref{c2-1'}, 
\eqref{L2u-3}, and \eqref{Lnlogc-1}, respectively. 
This yields \eqref{LlogL-p-2}, \eqref{LlogL-p-2'}, and \eqref{LlogL-p-2''} 
with a constant independent of the initial data. Hence, the proof is complete.
\end{proof}

As a consequence of Lemma \ref{LlogL-p} and Sobolev's inequality, $c$ is proven to be uniformly bounded in $L^p$ for all $p\geq 1$ as in the following lemma.
\begin{lemma} \label{Lp-v}
    For any $p \geq  1$, there exists $C_1=C(p)>0$ such that 
    \begin{align}\label{Lp-v-1}
        \int_\Omega c^p(\cdot,t) \leq C_1 \qquad \text{for all }t\in (0,T_{\rm max}).
    \end{align}
    Moreover, if $T_{\rm max}= \infty$ then there exists $T>0$ such that 
     \begin{align} \label{Lp-v-2}
        \int_\Omega c^p(\cdot,t) \leq C_2 \qquad \text{for all }t\in (T, \infty),
    \end{align}
    where $C_2>0$ is independent of initial data.
\end{lemma}
\begin{proof}
   By Sobolev's inequality, we can find $C_1=C_1(p, \Omega)>0$ such that 
    \begin{align} \label{Lp-v.1}
        \int_\Omega c^p \leq C_1  \left ( \int_\Omega |\nabla c|^{2} \right )^{\frac{p}{2}}+C_1 \left ( \int_\Omega c \right )^p  \quad \text{for all }t \in (0,T_{\rm max}),
    \end{align}
    In view of \eqref{L1-1'} and \eqref{LlogL-p-1'}, we find that $\int_\Omega |\nabla c(\cdot,t)|^2$ and $\int_\Omega c(\cdot,t)$ are uniformly bounded in time, which together with \eqref{Lp-v.1} implies \eqref{Lp-v-1}. If $T_{\rm max}= \infty$, then we use the above estimate except for using  \eqref{L1-3'} and \eqref{LlogL-p-2'} instead of \eqref{L1-1'} and \eqref{LlogL-p-1'} to deduce \eqref{Lp-v-2}. The proof is now complete.
\end{proof}

As a consequence of the above lemmas, we can now establish a uniform $W^{1,2}$ boundedness for $u$ as follows:
\begin{lemma} \label{Lpu}
There exists $C_1>0$ (possibly depending on 
the initial data) such that
    \begin{equation} \label{Lpu-1}
        \int_\Omega |\nabla u(\cdot,t)|^2  \leq C_1 \qquad \text{for all }t\in (0,T_{\rm max}).
    \end{equation}
As a consequence, for all $p>2$ there exists $C_2=C_2(p)>0$(possibly depending on 
the initial data) such that
    \begin{align} \label{Lpu-2}
        \int_\Omega |u(\cdot,t)|^p \leq C_2 \qquad \text{for all }t\in (0,T_{\rm max}).
    \end{align}
  Moreover, if $T_{\rm max}= \infty$ then there exists $T>0$ such that
     \begin{equation} \label{Lpu-3}
        \int_\Omega |\nabla u(\cdot,t)|^2  \leq C_3 \qquad \text{for all }t>T,
    \end{equation}
and
    \begin{align} \label{Lpu-4}
        \int_\Omega |u(\cdot,t)|^p \leq C_4 \qquad \text{for all }t>T,
    \end{align}
    where $C_3$ and $C_4$ are positive constants independent of initial data.
\end{lemma}
\begin{proof}
    The proof of \eqref{Lpu-1} is completely similar to \cite{YM-2016}[Lemma 3.6]. To prove \eqref{Lpu-2}, we simply apply Sobolev's inequality to obtain that 
    \begin{align} \label{Lpu.1}
        \int_\Omega |u|^p \leq  C_1 \left ( \int_\Omega |\nabla u|^2 \right )^{\frac{p}{2}} +C_1 \left ( \int_\Omega |u|^2 \right )^{\frac{p}{2}} \leq C_2 \qquad \text{for all }t\in (0,T_{\rm max}), 
    \end{align}
    where $C_1>0$ and $C_2>0$. To prove \eqref{Lpu-3}, we continue to follow the argument in \cite{YM-2016}[Lemma 3.6]; however, we must carefully track the dependence of each constant on the initial data. Consequently, we provide a detailed proof here. Now suppose that $T_{\rm max}=\infty$. By taking the Helmholtz projection $\mathcal{P}$ to the third equation in \eqref{1}, and then testing the resulting expression against $Au$, we arrive at the following
\begin{align} \label{Lpu.2}
    \frac{1}{2} \frac{d}{dt}\int_\Omega |A^\frac{1}{2}u|^2
    &= -\int_\Omega |Au|^2 
    - \int_\Omega  \mathcal{P}((u\cdot \nabla)u) \cdot Au  + \int_\Omega  \mathcal{P}(n \nabla \phi) \cdot Au
    +\int_\Omega  \mathcal{P}f \cdot Au \notag \\
    &\leq -\frac{1}{4}\int_\Omega |Au|^2 + \int_\Omega |(u \cdot \nabla )u|^2 
    + C_\phi \int_\Omega n^2 + \int_\Omega |f|^2 ,
\end{align}
where $C_\phi>0$ depends only on $\|\nabla\phi\|_{L^\infty(\Omega)}$. From Lemma \ref{L2u}, one can find $t_0>0$ such that 
    \begin{align*}
        \int_\Omega |u(\cdot, t)|^2 \leq C_3 \qquad \text{for all }t>t_0, 
    \end{align*}
    where $C_3>0$ independent of initial data. Using this, the Gagliardo–Nirenberg interpolation inequality and Young's inequality leads to
    \begin{align} \label{Lpu.3}
         \int_\Omega |(u \cdot \nabla )u|^2 &\leq \left \| u \right \|^2_{L^\infty(\Omega)} \int_\Omega |\nabla u|^2 \notag \\
         &\leq C_4 \left ( \int_\Omega |u|^2 \right )^{\frac{1}{2}} \left ( \int_\Omega |Au|^2 \right )^\frac{1}{2}\int_\Omega |\nabla u|^2  \notag \\
         &\leq \frac{1}{4} \int_\Omega |Au |^2 + C_5 \left ( \int_\Omega |\nabla u|^2 \right )^2 \qquad \text{for all }t >t_0, 
    \end{align}
    where $C_4$ and $C_5$ are positive constants independent of initial data. Setting 
    \begin{align*}
        h_1(t) = 2C_5\int_\Omega |\nabla u(\cdot,t)|^2,
    \end{align*}
    and
    \begin{align*}
        h_2(t)= 2C_\phi \int_\Omega n^2(\cdot,t) + 2\int_\Omega |f(\cdot,t)|^2, 
    \end{align*}
    then from \eqref{LlogL-p-2''} and \eqref{L2u-5}, there exists $T\geq t_0$ such that
    \begin{align}\label{Lpu.4}
        \int_t ^{t+1} h_1(s)\, ds  \leq C_6 \qquad \text{for all }t>T,
    \end{align}
    and 
    \begin{align}\label{Lpu.5}
        \int_t ^{t+1} h_2(s)\, ds \leq C_7 \qquad \text{for all }t>T,
    \end{align}
    where $C_6$ and $C_7$ are positive constants independent of initial data. 
 From \eqref{Lpu.2} and \eqref{Lpu.3}, together with the identity 
$\int_\Omega |A^{1/2}u|^2 = \int_\Omega |\nabla u|^2$, it follows that
the differential inequality below holds
    \begin{align*}
        \frac{d}{dt}\int_\Omega |\nabla u(\cdot,t)|^2 \leq h_1(t)\int_\Omega |\nabla u(\cdot,t)|^2 +h_2(t) \qquad \text{for all }t>T.
    \end{align*}
By invoking \eqref{Lpu.4}, \eqref{Lpu.5}, and \eqref{L2u-5}, and then 
applying Lemma \ref{ODI'}, we directly obtain \eqref{Lpu-3}. After increasing $T$ if necessary so that \eqref{L2u-4} also holds 
for all $t>T$, Sobolev's inequality yields, for every $p>2$,
\[
    \int_\Omega |u(\cdot,t)|^p
    \leq C(p)\left(\int_\Omega |\nabla u(\cdot,t)|^2\right)^{\frac p2}
      +C(p)\left(\int_\Omega |u(\cdot,t)|^2\right)^{\frac p2}
    \leq C_8  \qquad \text{for all }t>T,
\]
where $C_8>0$ is independent of the initial data. This implies \eqref{Lpu-4}.
\end{proof}
\section{\texorpdfstring{ Boundedness of $n$ in $L^p$}{ L^p Boundedness of Solutions}} \label{S4'}

Before proving the $ L^p $ boundedness of $ u $, we first state the following auxiliary lemma.
\begin{lemma} \label{gradc}
For any $p>1$, there exist $C_1=C_1(p)>0$ independent of initial data and $C_2=C_2(p)>0$ such that
    \begin{align}\label{gradc-1}
    \frac{d}{dt}\int_\Omega|\nabla c|^{2p} + \int_\Omega|\nabla c|^{2p}+\frac{p-1}{p}\int_\Omega |\nabla |\nabla c|^p|^2 \leq C_1 \int_\Omega n^2 |\nabla c|^{2p-2}+C_2
\end{align}
for all $t\in (0, T_{\rm max})$. Moreover, if $T_{\rm max}= \infty$ then there exists $T>0$ such that \eqref{gradc-1} holds for all $t>T$ with for some positive constant $C_2$ independent of initial data. 
\end{lemma}

\begin{proof}
    Differentiating the functional $y(t):= \int_\Omega |\nabla c|^{2p}$ gives us that
    \begin{align} \label{gradc.1}
        y'(t) &= 2p \int_\Omega |\nabla c|^{2p-2} \nabla c \cdot \nabla c_t \notag \\
        &=2p \int_\Omega |\nabla c|^{2p-2} \nabla c \cdot \nabla  \left ( -u \cdot \nabla c +\Delta c- \alpha c +\beta n \right ) \qquad \text{for all }t\in (0,T_{\rm max}).
    \end{align}
    Thanks to the following identity
    \begin{align*}
        \nabla c \cdot \nabla \Delta c = \frac{1}{2} \Delta (|\nabla c|^2)- |D^2 c|^2,
    \end{align*}
   we obtain that 
    \begin{align} \label{gradc.2}
        2p \int_\Omega |\nabla c|^{2p-2} \nabla c \cdot \nabla \Delta c &= p \int_\Omega |\nabla c|^{2p-2} \Delta (|\nabla c|^2) - 2p \int_\Omega |\nabla c|^{2p-2}|D^2c|^2 \notag \\
        &= - \frac{4(p-1)}{p} \int_\Omega |\nabla |\nabla c|^p|^2 + p\int_{\partial \Omega} |\nabla c|^{2p-2}\frac{\partial |\nabla c|^2}{\partial \nu } \,dS \notag \\
        &\quad- 2p \int_\Omega |\nabla c|^{2p-2}|D^2c|^2.
    \end{align}
    The inequality $\frac{\partial |\nabla c|^{2}}{\partial \nu} \leq C_1 |\nabla c|^2 ,$ (see \cite{Souplet-13}[Lemma 4.2])  for some $C_1>0$ depending only on $\Omega$, implies that
\[
p\int_{\partial \Omega} \frac{\partial |\nabla c|^2 }{\partial \nu}|\nabla c|^{2p-2}\, dS \leq pC_1\int_{\partial \Omega }|\nabla c|^{2p} \, dS  .
\]
Setting $ g := |\nabla c|^p $ and applying the trace embedding theorem 
$ W^{1,1}(\Omega) \hookrightarrow L^1(\partial \Omega) $ together with 
Young’s inequality, we obtain positive constants $ C_2 $ and $ C_3 $ such that
\begin{align*}
    pC_1\int_{\partial \Omega} g^2 \, dS &\leq C_2 \int_\Omega g|\nabla g| +C_2\int_\Omega g^2 \notag \\
    &\leq \frac{p-1}{p} \int_\Omega |\nabla g|^2 +C_3\int_\Omega g^2.
\end{align*}
Therefore, we have
\begin{align}  \label{gradc.3}
    pC_1\int_{\partial \Omega }|\nabla c|^{2p} \, dS \leq \frac{p-1}{p} \int_{\Omega}|\nabla|\nabla c|^p|^2 +C_3 \int_\Omega |\nabla c|^{2p}.  
\end{align}
Integrating by parts and applying Young's inequality, we deduce that 
\begin{align} \label{gradc.4}
    2p \beta  \int_\Omega |\nabla c|^{2p-2} \nabla c \cdot \nabla n &=-2p\beta \int_\Omega n \nabla |\nabla c|^{2p-2} \cdot \nabla c -2p\beta \int_\Omega n |\nabla c|^{2p-2}\Delta c \notag \\
    &= -4\beta (p-1)\int_\Omega n|\nabla c|^{p-2} \nabla |\nabla c|^p \cdot \nabla c-2p\beta \int_\Omega n |\nabla c|^{2p-2}\Delta c \notag \\
    &\leq \frac{p-1}{p}\int_\Omega |\nabla |\nabla c|^p|^2 +p\int_\Omega |\nabla c|^{2p-2}|D^2c|^2 +C_4 \int_\Omega n^2 |\nabla c|^{2p-2},
\end{align}
where $C_4>0$ depends only on fixed parameters and $p$ and the last inequality holds due to the inequality $|\Delta c|^2 \leq 2 |D^2 c|^2$. Now, we apply Young's inequality to deal with the transport term as follows:
\begin{align} \label{gradc.5}
    -2p \int_\Omega |\nabla c|^{2p-2} \nabla c \cdot \nabla (u \cdot \nabla c) &= 2p \int_\Omega \nabla |\nabla c|^{2p-2} \cdot \nabla c ( u \cdot \nabla c) +2p \int_\Omega |\nabla c|^{2p-2} \Delta c (u \cdot \nabla c ) \notag \\
    &\leq 4(p-1) \int_\Omega |\nabla |\nabla c|^p| |\nabla c|^p |u| +2p \int_\Omega |\nabla c|^{2p-1}|\Delta c||u| \notag \\
    &\leq \frac{p-1}{2p}\int_\Omega |\nabla |\nabla c|^p|^2+C_5 \int_\Omega |u|^2 |\nabla c|^{2p}+ p \int_\Omega |\nabla c|^{2p-2}|D^2c|^2
\end{align}
where $C_5>0$ independent of initial data. By H\"older's inequality, the two-dimensional Gagliardo--Nirenberg interpolation 
inequality applied to $|\nabla c|^p$, Young's inequality, and the estimates 
\eqref{Lpu-2} and \eqref{LlogL-p-1'}, we have
\begin{align}\label{gradc.6}
    &(1+C_3-2p\alpha)_+ \int_\Omega |\nabla c|^{2p}
    + C_5 \int_\Omega |u|^2 |\nabla c|^{2p} \notag\\
    &\leq  \left((1+C_3-2p\alpha)_+|\Omega|^{\frac12}
    +C_5\|u\|_{L^4(\Omega)}^2\right)
    \left(\int_\Omega |\nabla c|^{4p}\right)^{\frac12} \notag\\
    &\leq  C_6  \left(\int_\Omega |\nabla |\nabla c|^p|^2\right)^{\frac{2p-1}{2p}}
    \left(\int_\Omega |\nabla c|^2\right)^{\frac12}
    +C_6\left(\int_\Omega |\nabla c|^2\right)^p \notag\\
    &\leq  \frac{p-1}{2p}\int_\Omega |\nabla |\nabla c|^p|^2+C_7 ,
\end{align}
where $C_6>0$ and $C_7>0$.
Collecting \eqref{gradc.1}--\eqref{gradc.6}, we arrive at
\begin{align*}
    y'(t)+y(t)+\frac{p-1}{p}\int_\Omega |\nabla |\nabla c|^p|^2
    \leq C_4 \int_\Omega n^2|\nabla c|^{2p-2}+C_7 \qquad \text{for all }t\in(0,T_{\rm max}).
\end{align*}
Now, if $T_{\rm max}= \infty$ then from \eqref{Lpu-4} and \eqref{LlogL-p-2'}, one can find $T>0$ such that 
\begin{align*}
    \int_\Omega |u(\cdot,t)|^4 \leq C_8 \qquad \text{for all }t>T,
\end{align*}
and 
\begin{align*}
    \int_\Omega |\nabla c(\cdot,t)|^2 \leq C_9 \qquad \text{for all }t>T,
\end{align*}
where $C_8>0$ and $C_9>0$ independent of initial data. Thereupon in view of \eqref{gradc.6} we obtain that 
\begin{align*}
    y'(t) + y(t)+\frac{p-1}{p}\int_\Omega |\nabla |\nabla c|^p|^2 \leq C_4 \int_\Omega n^2 |\nabla c|^{2p-2}+C_{10} \qquad \text{for all } t >T, 
\end{align*}
where $C_{10}>0$ independent of initial data. The proof is now complete.
\end{proof}

We are now in a position to derive a key estimate in the proof of our main results, namely the $L^p$ boundedness of $n$ and, more importantly, the boundedness of $\int_\Omega n^p c^{-q}$ for all $p > 1$ and $0 < q < p - 1$. This estimate will be essential in establishing the $L^\infty$ boundedness of $n$ in Lemma~\ref{c-l}.

\begin{lemma} \label{Lp}
    For any $p>1$ and $0<q<p-1$, there exists $C_1>0$ such that 
    \begin{align} \label{Lp-1}
        \int_\Omega n^p(\cdot,t) + \int_\Omega n^p(\cdot,t)c^{-q}(\cdot,t) + \int_\Omega |\nabla c  (\cdot,t)|^{2p} \leq C_1 \qquad \text{for all } t\in (0,T_{\rm max}).
    \end{align}
    Moreover, if $T_{\rm max}= \infty$ then there exists $T>0$ such that
     \begin{align} \label{Lp-2}
        \int_\Omega n^p(\cdot,t) + \int_\Omega n^p(\cdot,t)c^{-q}(\cdot,t) + \int_\Omega |\nabla c(\cdot,t)|^{2p} \leq C_2 \qquad \text{for all } t>T,
    \end{align}
    where $C_2>0$ independent of initial data.
\end{lemma}
\begin{proof}
    By using the first equation of \eqref{1} and integration by parts, we obtain that
        \begin{align}\label{Lp'.1}
            \frac{d}{dt} \int_\Omega n^p &= p \int_\Omega n^{p-1} \left ( \Delta n -\chi \nabla \cdot \left ( n  \frac{\nabla c}{c^k} \right ) +rn -\frac{\mu n^2}{\log^\eta(n+e)} \right ) \notag \\
            &= -\frac{4(p-1)}{p}\int_\Omega |\nabla n^\frac{p}{2}|^2 + 2\chi(p-1) \int_\Omega n^\frac{p}{2}c^{-k} \nabla n^\frac{p}{2} \cdot \nabla c \notag \\
            &\quad+ rp \int_\Omega n^p -\mu p \int_\Omega \frac{n^{p+1}} {\log^\eta(n+e)}\qquad \text{for all } t\in (0,T_{\rm max}).
        \end{align}
    By direct calculations and several applications of integration by parts, we find that
       \begin{align} \label{Lp'.2}
        \frac{d}{dt} \int_\Omega n^p c^{-q} &= p \int_\Omega n^{p-1}c^{-q}n_t -q \int_\Omega n^p c^{-q-1}c_t \notag \\
        &= p \int_\Omega n^{p-1}c^{-q} \left ( - u \cdot \nabla n+ \Delta n -\chi \nabla \cdot \left ( nc^{-k} \nabla c \right ) ++rn-\frac{\mu n^2}{\log^\eta(n+e)} \right ) \notag \\
        &\quad -q \int_\Omega n^p c^{-q-1} \left ( -u \cdot \nabla c + \Delta c- \alpha c +\beta n \right ) \notag \\
        &=- p(p-1)\int_\Omega n^{p-2}c^{-q}|\nabla n|^2 +2pq \int_\Omega n^{p-1}c^{-q-1} \nabla n \cdot \nabla c  \notag\\
        &\quad + p(p-1)\chi \int_\Omega n^{p-1}c^{-q-k} \nabla n \cdot \nabla c -pq\chi \int_\Omega n^pc^{-q-k-1}|\nabla c|^2 \notag\\
        &\quad+ (rp+q\alpha) \int_\Omega n^p c^{-q} -q(q+1)\int_\Omega n^p c^{-q-2}|\nabla c|^2  \notag \\
        &\quad  - q\beta \int_\Omega n^{p+1}c^{-q-1} -\mu p \int_\Omega \frac{n^{p+1}c^{-q}}{\log^\eta(n+e)}
                \end{align}    
        for all $t\in (0,T_{\rm max})$, where the transport terms cancel because
\[
0=\int_\Omega u\cdot\nabla(n^p c^{-q})
=p\int_\Omega n^{p-1}c^{-q}u\cdot\nabla n
-q\int_\Omega n^p c^{-q-1}u\cdot\nabla c .
\]      The condition $0<q<p-1$ entails that $\frac{p^2q^2}{p(p-1)}< q(q+1)$, which allows us to choose $\varepsilon_1 $ sufficiently small such that
    \[
    \frac{p^2q^2}{p(p-1)-\varepsilon_1} +2\varepsilon_1 -q(q+1)  \leq 0.
    \]
    Applying Young's inequality with $\varepsilon_1 > 0$, we obtain
    \begin{align}\label{Lp'.3}
        2pq \int_\Omega n^{p-1}c^{-q-1} \nabla n \cdot \nabla c \leq (p(p-1)-\varepsilon_1) \int_\Omega n^{p-2}c^{-q} |\nabla n|^2+ \frac{p^2q^2}{p(p-1)-\varepsilon_1} \int_\Omega n^p c^{-q-2}|\nabla c|^2,
    \end{align}
    and 
    \begin{align}\label{Lp'.4}
        p(p-1)\chi \int_\Omega n^{p-1}c^{-q-k} \nabla n \cdot \nabla c &\leq \varepsilon_1 \int_\Omega n^{p-2} c^{-q}|\nabla n|^2 + \frac{p^2(p-1)^2\chi^2}{4\varepsilon_1}\int_\Omega n^p c^{-q-2k}|\nabla c|^2 \notag \\
        &\leq \varepsilon_1 \int_\Omega n^{p-2} c^{-q}|\nabla n|^2 + \varepsilon_1 \int_\Omega n^p c^{-q-2}|\nabla c|^2 +C_1 \int_\Omega n^p|\nabla c|^2,
    \end{align}
    where $C_1=C_1(\varepsilon_1)>0$, and
    \begin{align}\label{Lp'.5}
        2\chi(p-1) \int_\Omega n^\frac{p}{2} c^{-k} \nabla n^\frac{p}{2} \cdot \nabla c &\leq \frac{p-1}{p} \int_\Omega |\nabla n^\frac{p}{2}|^2 + p(p-1)\chi^2 \int_\Omega n^p c^{-2k} |\nabla c|^2 \notag\\
        &\leq  \frac{p-1}{p} \int_\Omega |\nabla n^\frac{p}{2}|^2 +\varepsilon_1 \int_\Omega n^p c^{-q-2} |\nabla c|^2 +C_2 \int_\Omega n^p|\nabla c|^2,
    \end{align}
    where $C_2=C_2(\varepsilon_1)>0$.
    Using Young's inequality again and Lemma \ref{Lp-v}, we infer that 
    \begin{align}\label{Lp'.6}
        (1+rp+q\alpha) \int_\Omega n^p c^{-q} &\leq q\beta \int_\Omega n^{p+1}c^{-q-1}+C_3 \int_\Omega c^{p-q} \notag \\
         &\leq q\beta \int_\Omega n^{p+1}c^{-q-1}+C_4,
    \end{align}
    for some $C_3>0$ and $C_4>0$, and
    \begin{align}\label{Lp'.7}
        (rp+1) \int_\Omega n^p \leq \mu p \int_\Omega \frac{n^{p+1}}{\log^\eta(n+e)} + C_5,
    \end{align}
    where $C_5>0$. By Lemma \ref{gradc}, we can find positive constants $C_6$ and $C_7$ such that
     \begin{align}\label{Lp'.7'}
    \frac{d}{dt}\int_\Omega|\nabla c|^{2p} + \int_\Omega|\nabla c|^{2p}+\frac{p-1}{p}\int_\Omega |\nabla |\nabla c|^p|^2 \leq C_6 \int_\Omega n^2 |\nabla c|^{2p-2}+C_7 \qquad \text{for all } t\in (0,T_{\rm max}). 
\end{align}
    Applying estimate \eqref{LlogL-p-1'} and the Gagliardo–Nirenberg interpolation inequality yields
    \begin{align}\label{Lp'.8}
        \int_\Omega |\nabla c|^{2p+2} &\leq C_8 \int_\Omega |\nabla |\nabla c|^p|^2 \int_\Omega |\nabla c|^2 +C_8 \left (\int_\Omega |\nabla c|^2 \right )^2 \notag \\
        &\leq C_9\int_\Omega |\nabla |\nabla c|^p|^2 +C_{10},
    \end{align}
    where $C_8$, $C_9$, and $C_{10}$ are positive constants. Employing Young's inequality and using \eqref{Lp'.8}, we get
    \begin{align}\label{Lp'.9}
        (C_1+C_2)\int_\Omega n^p |\nabla c|^2 + C_6\int_\Omega n^2 |\nabla c|^{2p-2} &\leq \frac{p-1}{pC_9} \int_\Omega |\nabla c|^{2p+2} +C_{11}\int_\Omega n^{p+1} \notag \\
        &\leq \frac{p-1}{p}\int_\Omega |\nabla |\nabla c|^p|^2 +C_{11}\int_\Omega n^{p+1}+C_{12},
    \end{align}
    where $C_{11}>0$ and $C_{12}>0$. Let 
    \begin{align*}
        y(t)= \int_\Omega n^p + \int_\Omega n^pc^{-q}+ \int_\Omega |\nabla c|^{2p} \qquad \text{for }t\in (0,T_{\rm max}).
    \end{align*}
   Collecting the above estimates, we arrive at
\begin{align}\label{Lp'.10}
   y'(t) +y(t) 
   &\leq  \left ( \frac{p^2q^2}{p(p-1)-\varepsilon_1} 
   +2\varepsilon_1 -q(q+1) \right )
   \int_\Omega n^p c^{-q-2}|\nabla c|^2 \notag\\
   &\quad -\frac{3(p-1)}{p} \int_\Omega |\nabla n^\frac{p}{2}|^2  
   +C_{11}\int_\Omega n^{p+1}+C_{13} \notag\\
   &\leq -\frac{3(p-1)}{p} \int_\Omega |\nabla n^\frac{p}{2}|^2  
   +C_{11}\int_\Omega n^{p+1}+C_{13}  \qquad \text{for }t\in (0,T_{\rm max}),
\end{align}
where $C_{13}=C_4+C_5+C_7+C_{12}$. By \eqref{L1-1} and \eqref{LlogL-p-1}, and since 
$s\log(s+e)\leq s\log s+C(1+s)$ for $s\geq0$, we can find $C_{14}>0$ such that
\begin{align}\label{Lp'.10'}
    \int_\Omega n\log(n+e) \leq C_{14} \qquad \text{for all }t\in(0,T_{\rm max}).
\end{align}
We apply Lemma \ref{C52.ILGN} with $m=p$, $\sigma=1$, $\xi=0$, and
\[
\varepsilon_2=\frac{3(p-1)}{pC_{11}C_{14}},
\]
to obtain
\begin{align}\label{Lp'.11}
    C_{11}\int_\Omega n^{p+1}
    &\leq \varepsilon_2 C_{11} \left(\int_\Omega n\log(n+e)\right)
    \int_\Omega |\nabla n^{\frac p2}|^2 \notag\\
    &\quad +\varepsilon_2 C_{11} \left(\int_\Omega n\right)^p
    \left(\int_\Omega n\log(n+e)\right) + C_{15} \notag\\
    &\leq \frac{3(p-1)}{p}  \int_\Omega |\nabla n^{\frac p2}|^2+C_{16},
\end{align}
    for some positive constants $C_{15}$ and $C_{16}$. Combining \eqref{Lp'.10} and \eqref{Lp'.11}, we obtain 
    \begin{align*}
          y'(t) +y(t) &\leq C_{17} \qquad \text{for all } t\in (0,T_{\rm max}),
    \end{align*}
    where $C_{17}=C_{13}+C_{16}$. Next, applying Gronwall’s inequality to the above estimate yields
\[
y(t) \leq \max\left\{ y(0), C_{17} \right\} \qquad \text{for all } t \in (0, T_{\max}).
\]
If $T_{\rm max}=\infty$, choose $t_0>0$ sufficiently large that
\eqref{L1-3}, \eqref{Lp-v-2} applied with exponent $p-q$,
\eqref{LlogL-p-2}, \eqref{LlogL-p-2'}, and the large-time estimate
in Lemma \ref{gradc} hold for all $t>t_0$, with constants independent
of the initial data. Repeating the preceding argument on $(t_0,\infty)$, we obtain
\[
    y'(t)+y(t)\leq C_{18} \qquad\text{for all }t>t_0,
\]
where $C_{18}>0$ is independent of the initial data. This entails that
\[
y(t) \leq e^{-(t - t_0)} y(t_0) + \left(1 - e^{-(t - t_0)}\right) C_{18} \qquad \text{for all } t > t_0.
\]
Finally, we choose $ T > t_0 $ such that $ e^{-(T - t_0)} y(t_0) < 1 $, which implies \eqref{Lp-2}. The proof is now complete.
\end{proof}

\section{Global boundedness, existence of absorbing sets and mass persistence}\label{S5}
In this section we first establish the global solvability and boundedness of solutions to the system \eqref{1}, then prove the existence of absorbing balls, and finally provide the proof for the mass persistence of solutions. 

\subsection{Fractional Stokes regularity and an absorbing radius}

The lower-order estimates obtained above already bound the forcing in the
projected fluid equation in a space $L^q(\Omega)$ with $q<2$.  The
$L^q$--$L^2$ smoothing of the Stokes semigroup therefore yields fractional
regularity after every positive time.  Once the lower-order bounds become
independent of the initial data, the same calculation gives an absorbing
radius in a fractional Stokes domain.

\begin{lemma}\label{uLinf}
Let $\gamma\in\left(\frac12,1\right)$ be as in \eqref{initial'} and take
\[
    \tau:=\min\left\{1,\frac{T_{\rm max}}6\right\}.
\]
Then there exists $C_1>0$, possibly depending on the initial data, such that
\begin{align}\label{uLinf-1}
    \|A^\gamma u(\cdot,t)\|_{L^2(\Omega)}
    \leq C_1 \qquad\text{for all }t\in(\tau,T_{\rm max}).
\end{align}
Consequently, for every $\theta\in(0,2\gamma-1)$ there exists $C_2>0$,
possibly depending on the initial data, such that
\begin{align}\label{uLinf-2}
    \|u(\cdot,t)\|_{C^{0,\theta}(\bar\Omega;\mathbb R^2)}
    \leq C_2 \qquad \text{for all }t\in(\tau,T_{\rm max}).
\end{align}
If $T_{\rm max}=\infty$, then, for every fixed
$\theta\in(0,2\gamma-1)$, there exist constants
$R_\gamma,R_\theta>0$, independent of the initial data, such that every
solution admits an entry time $T_u>0$ for which
\begin{align}\label{uLinf-3}
    \sup_{t>T_u}
    \left(
        \|u(\cdot,t)\|_{L^2(\Omega)}
        +\|A^\gamma u(\cdot,t)\|_{L^2(\Omega)}
    \right)
    \leq R_\gamma
\end{align}
and
\begin{align}\label{uLinf-4}
    \sup_{t>T_u}
    \|u(\cdot,t)\|_{C^{0,\theta}(\bar\Omega;\mathbb R^2)}
    \leq R_\theta.
\end{align}
Thus the fluid component eventually enters a bounded subset of
$D(A^\gamma)$ whose radius is independent of the initial data.
\end{lemma}

\begin{proof}
For $1<q<\infty$, let
\[
    L^q_\sigma(\Omega)  :=\overline{\left\{
        \varphi\in C_c^\infty(\Omega;\mathbb R^2):
        \nabla\cdot\varphi=0
    \right\}}^{\,L^q(\Omega;\mathbb R^2)}.
\]
Let $\mathcal P_q$ be the Helmholtz projection onto $L^q_\sigma(\Omega)$,
let $A_q$ be the Dirichlet Stokes operator on this space, and write
$S_q(t):=e^{-tA_q}$.  As before, $A=A_2$.

Choose
\[
    p_\gamma:=\frac{2}{1-\gamma}, \qquad
    q_\gamma:=\frac{2}{2-\gamma}, \qquad \sigma_\gamma:=\frac{1+\gamma}{2}.
\]
Then $p_\gamma>4$, $q_\gamma\in(1,2)$, and
\[
    \frac1{q_\gamma}=\frac1{p_\gamma}+\frac12, \qquad
    \gamma+\frac1{q_\gamma}-\frac12 = \sigma_\gamma<1.
\]
The standard $L^{q_\gamma}$-$L^2$ estimates for the Stokes semigroup give,
for $0<r\leq1$,
\begin{align}\label{fluid-stokes-L2}
    \|S_{q_\gamma}(r)\mathcal P_{q_\gamma}g\|_{L^2(\Omega)}
    &\leq C r^{-(1/q_\gamma-1/2)}  \|g\|_{L^{q_\gamma}(\Omega)},
    \\
\label{fluid-stokes-smoothing}
    \|A^\gamma S_{q_\gamma}(r)\mathcal P_{q_\gamma}g\|_{L^2(\Omega)}
    &\leq C r^{-\sigma_\gamma} \|g\|_{L^{q_\gamma}(\Omega)}.
\end{align}
To justify the second estimate explicitly, let
$h:=S_{q_\gamma}(r/2)\mathcal P_{q_\gamma}g$.  By
\eqref{fluid-stokes-L2}, $h\in L^2_\sigma(\Omega)$, and consistency of the
Stokes semigroups yields
\[
    S_{q_\gamma}(r)\mathcal P_{q_\gamma}g =S_2(r/2)h.
\]
Hence
\[
\begin{aligned}
    \|A^\gamma S_{q_\gamma}(r)\mathcal P_{q_\gamma}g\|_2
    &\leq C r^{-\gamma}\|h\|_2
    \\
    &\leq C r^{-\gamma-(1/q_\gamma-1/2)}\|g\|_{q_\gamma},
\end{aligned}
\]
which is \eqref{fluid-stokes-smoothing}; see, for instance,
\cite{Giga1986,Henry1981}.

By \eqref{Lpu-1} and \eqref{Lp-1}, used with $p=2$ and $q=\frac12$,
\[
    M_0:=\sup_{0<s<T_{\rm max}}
    \left(
        \|\nabla u(\cdot,s)\|_{L^2(\Omega)}
        +\|n(\cdot,s)\|_{L^2(\Omega)}
    \right)<\infty.
\]
Since $u=0$ on $\partial\Omega$, the Poincar\'e inequality and the
 two-dimensional Sobolev embedding imply
\begin{align}\label{fluid-u-lower}
    \|u(\cdot,s)\|_{L^2(\Omega)}
    +\|u(\cdot,s)\|_{L^{p_\gamma}(\Omega)}
    \leq C M_0 \qquad\text{for }0<s<T_{\rm max}.
\end{align}
Let
\[
    \mathcal G:=-(u\cdot\nabla)u+n\nabla\phi+f.
\]
Because $1/q_\gamma=1/p_\gamma+1/2$ and
$L^2(\Omega)\hookrightarrow L^{q_\gamma}(\Omega)$, H\"older's inequality
and \eqref{fluid-u-lower} give
\begin{align}\label{fluid-force-bound}
    \|\mathcal G(\cdot,s)\|_{L^{q_\gamma}(\Omega)}
    &\leq \|u(\cdot,s)\|_{L^{p_\gamma}(\Omega)}
    \|\nabla u(\cdot,s)\|_{L^2(\Omega)}  \notag\\
    &\quad  +C\|\nabla\phi\|_{L^\infty(\Omega)}
        \|n(\cdot,s)\|_{L^2(\Omega)}
    +C\|f\|_{L^\infty(\Omega\times(0,\infty))} \notag\\
    &\leq K_0
\end{align}
for all $s\in(0,T_{\rm max})$, where $K_0>0$ is finite.

Fix $t\in(\tau,T_{\rm max})$.  The mild formula for the projected fluid
equation holds in $L^{q_\gamma}_\sigma(\Omega)$:
\[
u(\cdot,t) =S_{q_\gamma}(\tau)u(\cdot,t-\tau) +\int_{t-\tau}^{t}
S_{q_\gamma}(t-s)\mathcal P_{q_\gamma} \mathcal G(\cdot,s)\,ds.
\]
By consistency, the first term equals
$S_2(\tau)u(\cdot,t-\tau)$.  Estimate \eqref{fluid-stokes-L2} shows that
the Duhamel integral converges in $L^2_\sigma(\Omega)$, whereas
\eqref{fluid-stokes-smoothing}, together with $\sigma_\gamma<1$ and the
closedness of $A^\gamma$, shows that this integral belongs to
$D(A^\gamma)$.  Therefore
\begin{align*}
    \|A^\gamma u(\cdot,t)\|_{L^2(\Omega)}
    &\leq C\tau^{-\gamma} \|u(\cdot,t-\tau)\|_{L^2(\Omega)}   +CK_0\int_0^\tau r^{-\sigma_\gamma}\,dr \\
    &\leq C\tau^{-\gamma}M_0   +\frac{C\tau^{1-\sigma_\gamma}}  {1-\sigma_\gamma}K_0.
\end{align*}
This yields \eqref{uLinf-1}.

The standard fractional-domain and Sobolev embeddings yield
\[
    D(A^\gamma)
    \hookrightarrow W^{2\gamma,2}(\Omega;\mathbb R^2)
    \hookrightarrow C^{0,\theta}(\bar\Omega;\mathbb R^2), \qquad 0<\theta<2\gamma-1,
\]
and hence
\[
    \|v\|_{C^{0,\theta}(\bar\Omega;\mathbb R^2)}
    \leq C_{\gamma,\theta} \bigl(\|v\|_{L^2(\Omega)}+\|A^\gamma v\|_{L^2(\Omega)}\bigr) \qquad\text{for }v\in D(A^\gamma).
\]
Together with \eqref{fluid-u-lower}, this implies \eqref{uLinf-2}.

Now assume that $T_{\rm max}=\infty$.  By \eqref{Lpu-3} and
\eqref{Lp-2}, again with $p=2$ and $q=\frac12$, there exist an entry time
$T_0>0$ and a constant $M_*>0$, independent of the initial data, such that
\[
    \|\nabla u(\cdot,s)\|_{L^2(\Omega)} + \|n(\cdot,s)\|_{L^2(\Omega)}
    \leq M_* \qquad\text{for all }s>T_0.
\]
Repeating \eqref{fluid-u-lower}--\eqref{fluid-force-bound}, we obtain
$K_*>0$, independent of the initial data, such that
\[
    \|\mathcal G(\cdot,s)\|_{L^{q_\gamma}(\Omega)}  \leq K_* \qquad\text{for all }s>T_0.
\]
For $t>T_0+1$, we apply the above estimate on $[t-1,t]$ obtaining
\[
    \|A^\gamma u(\cdot,t)\|_{L^2(\Omega)}  \leq C M_*+\frac{C}{1-\sigma_\gamma}K_*.
\]
The Poincar\'e inequality also yields
$\|u(\cdot,t)\|_{L^2(\Omega)}\leq C M_*$ for $t>T_0$.  Thus, after
increasing the constant if necessary,
\[
    \|u(\cdot,t)\|_{L^2(\Omega)} + \|A^\gamma u(\cdot,t)\|_{L^2(\Omega)}
    \leq R_\gamma \qquad\text{for all }t>T_0+1,
\]
where $R_\gamma>0$ is independent of the initial data.  This implies
\eqref{uLinf-3} with $T_u:=T_0+1$.  The above fractional-domain
embedding yields \eqref{uLinf-4} with
\[
    R_\theta:=C_{\gamma,\theta}R_\gamma.
\]
\end{proof}

With the help of Lemma \ref{Lp} and Lemma \ref{uLinf}, we can now derive a pointwise bound for $\nabla c$.
\begin{lemma}\label{W1,inf}
   There exists  $C_1>0$ such that 
    \begin{align}\label{W1,inf-1}
        \left \| c(\cdot,t) \right \|_{W^{1,\infty}(\Omega)} \leq C_1 \quad \text{for all } t\in (2\tau , T_{\rm max}),
    \end{align}
     where $\tau = \min \left \{ 1, \frac{T_{\rm max}}{6} \right \}$. Moreover, if $T_{\rm max}= \infty$ then there exists $C_2>0$ independent of initial data such that 
      \begin{align}\label{W1,inf-2}
      \limsup_{t\to \infty}  \left \| c(\cdot,t) \right \|_{W^{1,\infty}(\Omega)} \leq C_2.
    \end{align}
\end{lemma}

\begin{proof}
    Since the proof of \eqref{W1,inf-1}  proceeds along the same lines as that of Lemma 3.12 in \cite{YM-2016}, we omit the details here to avoid redundancy. In case $ T_{\max} = \infty $, we define
\[
M(T) := \sup_{t > T} \left\| \nabla c(\cdot, t) \right\|_{L^\infty(\Omega)} \quad \text{for all } T > 0.
\]
It follows from \eqref{W1,inf-1} that $M(T)<\infty$ for all $T>0$.
For $t>t_0$, by the variation-of-constants formula,
\[
\begin{aligned}
\nabla c(\cdot,t) &=\nabla e^{(t-t_0)(\Delta-\alpha)}c(\cdot,t_0)\\
&\quad + \beta\int_{t_0}^t
\nabla e^{(t-s)(\Delta-\alpha)}n(\cdot,s)\,ds\\
&\quad - \int_{t_0}^t \nabla e^{(t-s)(\Delta-\alpha)}
\bigl(u(\cdot,s)\cdot\nabla c(\cdot,s)\bigr)\,ds .
\end{aligned}
\]
Consequently,
\begin{align}\label{W1,inf.1}
\|\nabla c(\cdot,t)\|_{L^\infty(\Omega)}
&\leq C_1\bigl(1+(t-t_0)^{-1}\bigr)e^{-\alpha(t-t_0)}
\|c(\cdot,t_0)\|_{L^2(\Omega)}\notag\\
&\quad + C_1\beta\int_{t_0}^t
\bigl(1+(t-s)^{-\frac34}\bigr)e^{-\alpha(t-s)}
\|n(\cdot,s)\|_{L^4(\Omega)}\,ds\notag\\
&\quad + C_1\int_{t_0}^t
\bigl(1+(t-s)^{-\frac34}\bigr)e^{-\alpha(t-s)}
\|u(\cdot,s)\cdot\nabla c(\cdot,s)\|_{L^4(\Omega)}\,ds
\end{align}
for all $t>t_0+1$. From Lemmas \ref{Lp-v}, \ref{LlogL-p},
\ref{Lp}, and \ref{uLinf}, we can find $T>0$ such that

\begin{align}\label{W1,inf.2}
    \sup_{t>T} \left \{ \left \| c(\cdot,t) \right \|_{L^2(\Omega)} + \left \| \nabla c(\cdot,t) \right \|_{L^2(\Omega)}+ \left \| n(\cdot,t) \right \|_{L^4(\Omega)} + \left \|u(\cdot,t) \right \|_{L^\infty(\Omega)} \right \} \leq C_2,
\end{align}
where $C_2>0$ is independent of initial data. We also notice that 
\begin{align}\label{W1,inf.3}
    \left \|u(\cdot,s) \cdot \nabla c(\cdot,s) \right \|_{L^4(\Omega)} &\leq \left \| u(\cdot,s) \right \|_{L^\infty(\Omega)} \left \| \nabla c(\cdot,s) \right \|_{L^4(\Omega)} \notag \\
    &\leq C_2 \left \| \nabla c(\cdot,s) \right \|^\frac{1}{2}_{L^\infty(\Omega)}\left \| \nabla c(\cdot,s) \right \|^\frac{1}{2}_{L^2(\Omega)} \notag \\
    &\leq C_2^{\frac{3}{2}} M^\frac{1}{2}(T) \qquad \text{for all }s>T.
\end{align}
Choosing $t_0=T$ and using \eqref{W1,inf.2} and \eqref{W1,inf.3}, we obtain
\[
    \|\nabla c(\cdot,t)\|_{L^\infty(\Omega)} \leq C_3+C_4M(T)^{\frac12} \qquad\text{for all }t>T+1,
\]
where $C_3,C_4>0$ are independent of the initial data. Consequently,
\[
    M(T+1)\leq C_3+C_4M(T)^{\frac12}.
\]
 Since $\lim_{T \to \infty} M(T) = M < \infty$ by \eqref{W1,inf-1}, letting $T \to \infty$ in the inequality, we obtain
\[
M \leq C_3 + C_4 M^{1/2}.
\]
Hence $M \leq \max\{2C_3,\,(2C_4)^2\}$, which is  \eqref{W1,inf-2}. This completes the proof.
\end{proof}

With the aid of Lemmas~\ref{Lp} and~\ref{uLinf}, we now establish a $ C^{0,\lambda} $ bound for $ n $, for some $ \lambda \in (0,1) $, by exploiting the regularizing properties of the Neumann heat semigroup.

\begin{lemma} \label{c-l}
    There exist $\lambda \in (0,1)$ and $C_1>0$ such that 
    \begin{align} \label{c-l-1}
        \left \| n(\cdot,t) \right \|_{C^{0,\lambda}(\bar{\Omega})} \leq C_1 \qquad \text{for all } t\in (3\tau , T_{\rm max}),
    \end{align}
    where $\tau = \min \left \{ 1, \frac{T_{\rm max}}{6} \right \}$. Moreover, if $T_{\rm max}= \infty$ then there exists $C_2>0$  independent of initial data such that 
     \begin{align} \label{c-l-2}
       \limsup_{t \to \infty} \left \| n(\cdot,t) \right \|_{C^{0,\lambda}(\bar{\Omega})} \leq C_2 .
    \end{align}
    
\end{lemma}

\begin{proof}
{First of all, if $p>\max\left\{2,\frac{1}{1-k}\right\}$, then there exists $C_1>0$ such that}    
    \begin{align} \label{c-l.1}
        \int_\Omega \frac{n^p (\cdot,t) |\nabla c(\cdot,t)|^p}{c^{pk}(\cdot,t)} \leq C_1 \quad \text{for all }t\in (2\tau,T_{\rm max}).
    \end{align}
{
Fix $\lambda\in(0,1)$ and choose $\sigma\in\left(\frac{\lambda}{2},\frac12\right)$. 
We now take
\[
p>\max\left\{2,\frac{1}{1-k},\frac{2}{2\sigma-\lambda}\right\}.
\]
Let $B$ denote the realization of $-\Delta+1$ in $L^p(\Omega)$ under homogeneous 
Neumann boundary conditions. Since $2\sigma-\frac{2}{p}>\lambda$, we have
\[
D(B^\sigma)\hookrightarrow C^{0,\lambda}(\bar\Omega).
\]
For $t\in(3\tau,T_{\max})$, the variation-of-constants formula on $[t-\tau,t]$ and the standard smoothing estimates give
\[
\begin{aligned}
\|n(\cdot,t)\|_{C^{0,\lambda}(\bar\Omega)}
&\le C\tau^{-\sigma}\|n(\cdot,t-\tau)\|_{L^p(\Omega)} \\
&\quad + C\int_{t-\tau}^{t}(t-s)^{-\sigma-\frac12}
\left(
\left\|\chi n(\cdot,s)\frac{\nabla c(\cdot,s)}{c^k(\cdot,s)}\right\|_{L^p(\Omega)}
+\|n(\cdot,s)u(\cdot,s)\|_{L^p(\Omega)}
\right)\,ds \\
&\quad + C\int_{t-\tau}^{t}(t-s)^{-\sigma}
\left( \|n(\cdot,s)\|_{L^p(\Omega)}
+\|n(\cdot,s)\|_{L^{2p}(\Omega)}^2 \right)\,ds .
\end{aligned}
\]
Since $t-\tau>2\tau$ and $\sigma<\frac12$, \eqref{c-l.1}, Lemma \ref{Lp}, and Lemma \ref{uLinf} imply \eqref{c-l-1}. 
Now assume that $T_{\max}=\infty$.}
{
For $t>1$, by the variation-of-constants formula together with $\nabla\cdot u=0$,
\begin{align} \label{c-l.2}
    \left \| n(\cdot,t) \right \|_{C^{0,\lambda}(\bar{\Omega})} 
    &\leq C_2 \left \| B^\sigma n(\cdot,t) \right \|_{L^p(\Omega)} \notag \\
    &\leq C_2 \left \| B^\sigma e^{\Delta} n(\cdot,t-1) \right \|_{L^p(\Omega)} \notag \\
    &\quad + C_2 \int_{t-1}^t 
    \left \| B^\sigma e^{(t-s)\Delta} \nabla \cdot \left (
    \chi n(\cdot,s) \frac{\nabla c(\cdot,s)}{c^k(\cdot,s)}
    +n(\cdot,s)u(\cdot,s) \right ) \right \|_{L^p(\Omega)}\, ds \notag \\
    &\quad+C_2 \int_{t-1}^t 
    \left \| B^\sigma e^{(t-s)\Delta} \left \{ 
    rn(\cdot,s)-\frac{\mu n^2(\cdot,s)}{\log^\eta(n(\cdot,s)+e)}
    \right \} \right \|_{L^p(\Omega)}\, ds .
\end{align}}
    Invoking standard regularization properties of the analytic semigroup $(e^{-\sigma B})_{\sigma \geq 0}$ \cite{Friedman1969} and the estimate \eqref{Lp-2} in Lemma \ref{Lp} implies the existence of $T_1>0$ such that
    \begin{align} \label{c-l.3}
        C_2 \left \| B^\sigma e^{\Delta} n(\cdot,t-1) \right \|_{L^p(\Omega)} &= C_2 e  \left \| B^\sigma e^{-B} n(\cdot,t-1) \right \|_{L^p(\Omega)}  \notag \\
        &\leq C_3 \left \| n(\cdot,t-1) \right \|_{L^p(\Omega)} \leq C_4 \quad \text{for all }t>T_1,
    \end{align}
    where $C_3$ and $C_4$ are positive constants independent of initial data. We now make use of the classical smoothing properties of the Neumann heat semigroup $\left( e^{t\Delta} \right)_{t \geq 0}$ in $\Omega$ (cf.~\cite[Lemma~1.3]{Winkler_2010}) to deduce that
   \begin{align}\label{c-l.4}
    &C_2 \int_{t-1}^t \left\| B^\sigma e^{(t-s)\Delta} \nabla \cdot \left( \chi
    n(\cdot,s) \frac{\nabla c(\cdot,s)}{c^k(\cdot,s)} + n(\cdot,s) u(\cdot,s) 
    \right) \right\|_{L^p(\Omega)} \, ds \notag \\
    &= C_2 \int_{t-1}^t e^{\frac{t-s}{2}} 
    \left\| B^\sigma e^{-\frac{t-s}{2}B} e^{\frac{t-s}{2}\Delta} 
    \nabla \cdot \left(  \chi
    n(\cdot,s) \frac{\nabla c(\cdot,s)}{c^k(\cdot,s)} + n(\cdot,s) u(\cdot,s) 
    \right) \right\|_{L^p(\Omega)} \, ds \notag \\
    &\leq C_5 \int_{t-1}^t e^{\frac{t-s}{2}} \cdot \left ( \frac{t-s}{2} \right )^{-\sigma } \left \| e^{\frac{t-s}{2} \Delta}  \nabla \cdot \left( \chi
    n(\cdot,s) \frac{\nabla c(\cdot,s)}{c^k(\cdot,s)} + n(\cdot,s) u(\cdot,s) 
    \right)  \right \|_{L^p(\Omega)}\, ds \notag \\
    &\leq C_6 \int_{t-1}^t  e^{\frac{t-s}{2}} \left ( \frac{t-s}{2} \right )^{-\sigma -\frac{1}{2}} \left \{ \chi \left \| n(\cdot,s) \frac{\nabla c(\cdot,s)}{c^k(\cdot,s)} \right \|_{L^p(\Omega)}+ \left \| n (\cdot,s) u(\cdot,s) \right \|_{L^p(\Omega)} \right \} \, ds,
\end{align}
where $C_5>0$ and $C_6>0$. 
Using \eqref{W1,inf-2} in Lemma \ref{W1,inf} and \eqref{Lp-2} in Lemma \ref{Lp}, 
and recalling that $pk<p-1$, there exists $T_2>0$ such that
\begin{align}\label{c-l.5}
    \left \| \chi n(\cdot,s) \frac{\nabla c(\cdot,s)}{c^k(\cdot,s)} \right \|_{L^p(\Omega)}
    &\leq \chi \left \| \nabla c(\cdot,s) \right \|_{L^\infty(\Omega)}
    \left \| \frac{ n(\cdot,s)}{c^k(\cdot,s)} \right \|_{L^p(\Omega)}
    \leq C_7
\end{align}
for all $s>T_2$, where $C_7>0$ is independent of the initial data. 
Using \eqref{Lp-2} and \eqref{uLinf-4}, there exists $T_3>0$ such that
\begin{align}\label{c-l.6}
     \left \| n(\cdot,s) u(\cdot,s) \right \|_{L^p(\Omega)}
     \leq \left \| u(\cdot,s) \right \|_{L^\infty(\Omega)}
     \left \| n(\cdot,s)\right \|_{L^p(\Omega)} \leq C_8
\end{align}
for all $s>T_3$, where $C_8>0$ is independent of the initial data.
Collecting \eqref{c-l.4}, \eqref{c-l.5}, and \eqref{c-l.6}, we obtain 
\begin{align}\label{c-l.7}
     C_2 \int_{t-1}^t &\left\| B^\sigma e^{(t-s)\Delta} \nabla \cdot \left( 
    \chi n(\cdot,s) \frac{\nabla c(\cdot,s)}{c^k(\cdot,s)} + n(\cdot,s) u(\cdot,s) 
    \right) \right\|_{L^p(\Omega)} \, ds \notag \\
    &\leq C_6(C_7+C_8)e^{\frac{1}{2}} 2^{\sigma +\frac{1}{2}} 
    \int_{t-1}^t (t-s)^{-\sigma -\frac{1}{2}} \, ds \notag \\
    &\leq C_6(C_7+C_8)e^{\frac{1}{2}} 2^{\sigma +\frac{1}{2}} 
    \cdot \frac{1}{\frac{1}{2}-\sigma } := C_9
    \quad \text{for all }t>1+\max \left \{ T_2,T_3 \right \}.
\end{align} Proceeding analogously and using \eqref{Lp-2} with exponents 
$p$ and $2p$, we infer that there exists $T_4>0$ such that
\begin{align}\label{c-l.8}
    C_2 \int_{t-1}^t 
    &\left \| B^\sigma e^{(t-s)\Delta} \left \{ 
    rn(\cdot,s)-\frac{\mu n^2(\cdot,s)}{\log^\eta(n(\cdot,s)+e)}
    \right \} \right \|_{L^p(\Omega)}\,ds \notag \\
    &\leq C_{10} \int_{t-1}^t (t-s)^{-\sigma}
    \left(  r\left \| n(\cdot,s) \right \|_{L^p(\Omega)}
    +\mu \left \| n(\cdot,s) \right \|^2_{L^{2p}(\Omega)}
    \right)\, ds \notag \\
    &\leq C_{11} \int_{t-1}^t (t-s)^{-\sigma}\, ds
    \leq \frac{C_{11}}{1-\sigma} \qquad \text{for all } t>1+T_4 .
\end{align}
{
Collecting \eqref{c-l.2}, \eqref{c-l.3}, \eqref{c-l.7}, and \eqref{c-l.8}, we obtain
\begin{align*}
    \left \| n(\cdot,t) \right \|_{C^{0,\lambda} (\bar{\Omega})} 
    \leq C_{12} \qquad \text{for all }t> T,
\end{align*}
where
\[
C_{12}:=C_4+C_9+\frac{C_{11}}{1-\sigma}
\]
is independent of the initial data and
\[
T:=\max\left\{T_1,\,1+T_2,\,1+T_3,\,1+T_4\right\}.
\]
Sending $t\to\infty$ implies \eqref{c-l-2}.}
\end{proof}

\subsection{Proof of the main results}
We are now ready to prove the global existence and boundedness of solutions to the system \eqref{1}.
\begin{proof}[Proof of Theorem \ref{thm1}]
   Theorem \ref{thm1} is an immediate consequence of \eqref{c-l-1} in Lemma \ref{c-l}, \eqref{W1,inf-1} in Lemma \ref{W1,inf}, \eqref{uLinf-1} in Lemma \ref{uLinf}, and the extensibility property of solutions \eqref{ext} in Lemma \ref{local}.
\end{proof}
Next, the existence of an absorbing set with respect to topology in $C^0(\bar{\Omega})\times W^{1,\infty}(\Omega) \times C^{0, \theta}(\Omega)$ for some $\theta \in (0,1)$ is established as follows:
\begin{proof}[Proof of Theorem \ref{thm2}]
    By collecting the estimates \eqref{c-l-2}, \eqref{W1,inf-2}, and \eqref{uLinf-4}, we prove the desired result.
\end{proof}
Finally, we prove the last main result.

\begin{proof}[Proof of Theorem \ref{thm3}]
Let
\[
    m(t):=\frac{1}{|\Omega|}\int_\Omega n(x,t)\,dx .
\]
Integrating the first equation of \eqref{1} over $\Omega$, and using
$\nabla\cdot u=0$ together with the boundary conditions, we obtain
\begin{equation}\label{mass.1}
    m'(t) = rm(t)  -\frac{\mu}{|\Omega|}
    \int_\Omega\frac{n^2}{\log^\eta(n+e)} \qquad\text{for all }t>0.
\end{equation}

We first derive an upper estimate. Let
\[
    \psi(s):=\log^\eta(s+e),\qquad s\geq0.
\]
By a direct calculation we obtain
\[
    \psi''(s) =  \frac{\eta\log^{\eta-2}(s+e)} {(s+e)^2}
    \bigl(\eta-1-\log(s+e)\bigr)<0,
\]
and hence $\psi$ is concave. By the Cauchy--Schwarz and Jensen
inequalities,
\[
\begin{aligned}
    \left(\int_\Omega n\right)^2
    &\leq   \left(\int_\Omega\frac{n^2}{\psi(n)}\right) \left(\int_\Omega\psi(n)\right)\\
    &\leq |\Omega|\psi(m(t)) \int_\Omega\frac{n^2}{\psi(n)} .
\end{aligned}
\]
Consequently,
\[
    \frac{1}{|\Omega|}  \int_\Omega\frac{n^2}{\log^\eta(n+e)}
    \geq  \frac{m^2(t)}{\log^\eta(m(t)+e)}.
\]
Therefore,
\begin{equation}\label{mass.2}
    m'(t)  \leq  rm(t) -\frac{\mu m^2(t)}{\log^\eta(m(t)+e)}.
\end{equation}

Moreover,
\[
    \frac{d}{ds}  \left(\frac{s}{\log^\eta(s+e)}\right)
    = \frac{1}{\log^\eta(s+e)} \left( 1-\frac{\eta s}{(s+e)\log(s+e)}
    \right)>0.
\]
Since
\[
    \frac{s}{\log^\eta(s+e)} \to \infty \qquad\text{as }s\to\infty,
\]
there exists a unique $\overline m>0$ such that
\[
    \frac{\mu\overline m}{\log^\eta(\overline m+e)}=r.
\]

Let $y$ be the solution to
\[
    y'  = ry-\frac{\mu y^2}{\log^\eta(y+e)},  \qquad y(0)=m(0).
\]
Then \eqref{mass.2} implies that
\[
    m(t)\leq y(t)\qquad\text{for all }t\geq0.
\]
We also have
\[
y(t) \leq\max\{m(0),\overline m\}  \qquad\text{for all }t\geq0,
\]
and
\[
    y(t) \to \overline m  \qquad\text{as }t\to\infty.
\]
Thus
\begin{equation}\label{mass.3}
    m(t)\leq\max\{m(0),\overline m\} \qquad\text{for all }t\geq0,
\end{equation}
and
\begin{equation}\label{mass.4}
    \limsup_{t\to\infty}m(t)\leq\overline m.
\end{equation}

We next derive a lower estimate. Let $\lambda\in(0,1)$ and $C_2>0$ be as
in \eqref{c-l-2}, and let
\[
    K:=C_2+1.
\]
It follows from \eqref{c-l-2} that there exists $T>0$, possibly depending
on the initial data, such that
\begin{equation}\label{mass.5}
    \|n(\cdot,t)\|_{C^{0,\lambda}(\bar\Omega)} \leq K  \qquad\text{for all }t\geq T.
\end{equation}

Let
\[
    \theta:=\frac{\lambda}{2+\lambda}\in(0,1).
\]
By the Hölder interpolation inequality, there exists a constant $C_I>0$, depending
only on $\Omega$ and $\lambda$, such that
\[
    \|v\|_{L^\infty(\Omega)} \leq  C_I  \|v\|_{C^{0,\lambda}(\bar\Omega)}^{1-\theta}
    \|v\|_{L^1(\Omega)}^\theta
\]
for every nonnegative $v\in C^{0,\lambda}(\bar\Omega)$. Hence
\begin{equation}\label{mass.6}
    \|n(\cdot,t)\|_{L^\infty(\Omega)} \leq A m^\theta(t)  \qquad\text{for all }t\geq T,
\end{equation}
where
\[
    A:=C_IK^{1-\theta}|\Omega|^\theta.
\]
In particular, $A$ is independent of the initial data.

Since $\log^\eta(n+e)\geq1$, it follows from \eqref{mass.6} that
\[
\begin{aligned}
    \frac{1}{|\Omega|}  \int_\Omega\frac{n^2}{\log^\eta(n+e)}
    &\leq  \frac{1}{|\Omega|}\int_\Omega n^2\\
    &\leq    \|n(\cdot,t)\|_{L^\infty(\Omega)}m(t)\\
    &\leq    A m^{1+\theta}(t)
\end{aligned}
\]
for all $t\geq T$. Therefore,
\begin{equation}\label{mass.7}
    m'(t) \geq   rm(t)-\mu A m^{1+\theta}(t) \qquad\text{for all }t\geq T.
\end{equation}

Define
\[
    \underline m    :=  \left(\frac{r}{\mu A}\right)^{1/\theta}>0,
\]
and let $z$ be the solution to
\[
    z'=rz-\mu A z^{1+\theta}, \qquad z(T)=m(T).
\]
Since $m$ is a supersolution of this scalar equation, we have
\[
    m(t)\geq z(t) \qquad\text{for all }t\geq T.
\]
Moreover,
\[
    z(t)\geq\min\{m(T),\underline m\} \qquad\text{for all }t\geq T,
\]
and
\[
    z(t)\to \underline m  \qquad\text{as }t\to\infty.
\]
Consequently,
\begin{equation}\label{mass.8}
    \liminf_{t\to\infty}m(t)\geq\underline m.
\end{equation}

Finally, $m$ is continuous and strictly positive on $[0,T]$, because
$n_0\not\equiv0$ and $n(\cdot,t)>0$ for all $t>0$. Hence
\[
M_*  :=  |\Omega|   \min\left\{   \min_{0\leq s\leq T}m(s),\,  \underline m \right\}>0.
\]
Taking
\[
    \underline M:=|\Omega|\underline m,  \qquad   \overline M:=|\Omega|\overline m,
\]
and combining \eqref{mass.3}, \eqref{mass.4}, and \eqref{mass.8}
completes the proof.
\end{proof}

\section*{Declarations}
\paragraph{Conflict of Interest} The authors declare that they have no conflict of interest.
\paragraph{Acknowledgments} Minh Le was supported by the Hangzhou Postdoctoral Research Grant.

     \paragraph{Data Availability}
 Data sharing not applicable to this article as no datasets were generated or analyzed during
the current study.

\end{document}